\documentclass[11pt,a4paper]{article}

\usepackage{amsmath,amsthm,amsopn,amsfonts,amssymb,dsfont,color}
\usepackage{psfrag}
\usepackage{geometry}
\usepackage{graphicx}
\usepackage[utf8]{inputenc}
\usepackage[english]{babel}
\usepackage[active]{srcltx}
\usepackage{subcaption}

\usepackage{showlabels}

\usepackage{a4,a4wide}
\usepackage{color}
\usepackage{enumerate}

\def\ep{{\varepsilon}}
\def\spt{{\rm spt}}
\def\R{\mathbb R}

\newtheorem{theo}{\textbf{Theorem}}[section]

\newtheorem{lem}{\textbf{Lemma}}[section]
\newtheorem{prop}[theo]{\textbf{Proposition}}

\newtheorem{defi}{\textbf{Definition}}[section]

\newtheorem{rem}[theo]{\textbf{Remark}}

\numberwithin{equation}{section}

\title{Convergence to a terrace solution in multistable reaction-diffusion equation with discontinuities}
\date{\today}
\author{Thomas Giletti\footnote{Institut Elie Cartan de Lorraine, UMR 7502, University of Lorraine, 54506 Vandoeuvre-les-Nancy, France
thomas.giletti@univ-lorraine.fr}, Ho-Youn Kim\footnote{Department of Mathematical Sciences, KAIST, 291 Daehak-ro, Yuseong-gu, Daejeon, 34141, Korea}}

\begin{document}

\maketitle

\begin{abstract}
In this paper we address the large-time behavior of solutions of bistable and multistable reaction-diffusion equation with discontinuities around the stable steady states. We show that the solution always converges to a special solution, which may either be a traveling wave in the bistable case, or more generally a terrace (i.e. a collection of stacked traveling waves with ordered speeds) in the multistable case.
\end{abstract}

\section{Introduction}

In this paper, we investigate the large-time dynamics of solutions of the reaction-diffusion equation,
\begin{equation}\label{eq:rd}
\partial_t u = \partial_{xx} u + f(u),\quad  t >0, \ \  x\in\R,
\end{equation}
supplemented with the initial condition
$$u(t=0,\cdot) \equiv u_0 .$$
when the reaction function $f$ has several stable steady states and discontinuities. This is motivated by population dynamics and ecology, where the unknown function $u$ typically stands for a species density and discontinuities may arise from harvesting terms. In particular, such discontinuities may allow for finite time extinction, which is a realistic phenomena that cannot be observed in the usual case of smooth reaction. These discontinuous models have been considered in~\cite{Kim2020,KimChung} for the bistable case, and more recently by the same authors of the present paper for the more general multistable framework~\cite{GilettiKimKim_existence}. 

By multistable, we will mean that there exists a finite number of constant steady states, denoted $\theta_i$'s. More precisely,
\begin{eqnarray}
 && 0 = \theta_0 < \theta_1 < \cdots < \theta_{2I-2} < \theta_{2I-1} < \theta_{2I} = 1, \nonumber    \\
 &&  f(u)>0\ \text{ for }\ u<\theta_{0}\ \ \text{ and }\ \theta_{2i-1}<u<\theta_{2i},\quad i=1,\cdots,I,  \label{H1}\\
&& f(u)<0\ \text{ for }\ u>\theta_{2I}\ \text{ and }\ \theta_{2i}<u<\theta_{2i+1},\quad i=0,\cdots,I-1 ,\nonumber
\end{eqnarray}
with $I \in \mathbb{N}^*$. As far as the reaction term $f$ is concerned, we make the assumption that
\begin{equation}\label{H2}
f\in C^1(\R\setminus\{\theta_{2i}\,| \ 0\leq i \leq I \})\cap {\rm Lip}\, (\R\setminus\{\theta_{2i}\,| \ 0\leq i \leq I \}) \cap L^\infty (\mathbb{R}).
\end{equation}
while at the stable steady states $\theta_{2i}$'s it is discontinuous, i.e.
\begin{equation}\label{H3}
\lim_{u \to \theta_{2i}^-}f(u)>0>\lim_{u \to \theta_{2i}^+}f(u),\quad i=0,\cdots,I.
\end{equation}

Most of the mathematical literature for reaction-diffusion equations, in particular in the context of population dynamics, has been devoted to situations where instead $f$ is at least Lipschitz-continuous. A typical case is when $f$ is of the bistable type, i.e. it satisfies \eqref{H1} with $I=1$, which allows to take into account a so-called strong Allee effect, which is a positive correlation between the reproduction rate and the species density, at small densities. It is well-known that in such a situation equation \eqref{eq:rd} admits (unique up to shifts) traveling wave solutions connecting $0$ and $1$, which are special entire in time solutions $u(t,x)$ that can be written in the form
$$u(t,x) = \phi (x-ct),$$
where $c  \in \mathbb{R}$ is called the wave speed, and $\phi$ the wave profile satisfies
\[
\phi ' \leq 0 , \ \text{ with } \phi(z)\to1\ \text{ as }\ z\to-\infty,\ \text{ and }\ \phi(z)\to0\ \text{ as }\ z\to\infty.
\]
Furthermore, it is attractive in the sense that, for large classes of initial data including any initial datum with $0 \leq u_ 0 \leq 1$, $u_0 (-\infty) =1$ and $u_0 (+\infty) = 0$, the solution converges to a shift of this traveling wave as $t \to +\infty$, in the sense that
$$\exists X >0 , \qquad \lim_{t \to +\infty} \sup_{x \in \R} | u(t,x) - \phi (x-ct- X) | = 0.$$
We refer to the famous works~\cite{AW75,FifeMcLeod77} for more details, though the mathematical literature on the subject is much richer. Indeed, in the past decades, these seminal results have been a motivation for the construction of traveling wave solutions in a large range of reaction-diffusion models, including heterogeneous environments, competition or prey-predator systems, etc.

In the more intricate multistable case where $I >1$, in general there does not exist such a traveling wave, at least not connecting directly $0$ and $1$. This led to the notion of minimal decomposition~\cite{FifeMcLeod77}, which is a collection of traveling waves connecting steady states sequentially from $0$ to $1$. It typically provides a theoretical groundwork for the depiction of propagation phenomena  that takes place into several intermediate steps. More recently, this notion has been relabeled as ``propagating terrace'' in~\cite{DGM,GM,GilettiRossi,Polacik}, where it has been extended in particular to heterogeneous environments or reaction terms that allow an infinite number of steady states. We also refer to Definition~\ref{terracedef} below. In these aforementioned works, not only the existence of propagating terraces has been addressed, but also their attractiveness with respect to solutions of the Cauchy problem associated with \eqref{eq:rd}. Without going into the details, it was found that, as in the more standard bistable case, the propagating terrace truly dictates the large-time behavior of solutions. However, it does it piecewise due to the fact that it does not consist of a single but a family of special solutions of~\eqref{eq:rd}. \medskip

Now the novelty of the current manuscript is that, on the contrary to the above cited works, here we make the assumption that $f$ has discontinuities at the stable steady states. Actually, we briefly point out that we can allow $f$ to also have a finite number of discontinuitives away from steady states, and our arguments still apply with straightforward modification. Still it must be avoided that $f$ has discontinuities at the unstable steady states $\theta_{2i-1}$'s, as such a situation raises serious forward-in-time well-posedness issues (even in the pure ODE case with no diffusion).

In any case, under assumptions \eqref{H1}-\eqref{H2}-\eqref{H3}, we have established the existence and uniqueness (up to some shifts) of a propagating terrace in~\cite{GilettiKimKim_existence} using a phase plane analysis. We will recall these results more precisely below. Furthermore, due to the discontinuities of~$f$, we also found that all the traveling waves forming the propagating terrace are identical on some half-lines to their limiting (stable) stable steady states; we call such traveling waves ``compact''. In particular it is possible to ``glue'' them into a single special solution, which we refer to as a ``terrace solution''; see Definitions~\ref{terracedef} and~\ref{terracesolution}. This is the specificity of the discontinuous framework, which allows for a deeper analysis of the properties of the terrace. In this work, we continue our analysis by checking that solutions of the Cauchy problem~\eqref{eq:rd} do converge to a terrace solution as time goes to infinity.

\section{Definitions and main results}

In this section we introduce the key definitions and recall some important results, first from~\cite{Kim2020} on the well-posedness of the parabolic problem~\eqref{eq:rd}, second from~\cite{GilettiKimKim_existence} on the existence of terraces.

Throughout this work, we will assume that the initial condition $u_0$ satisfies
\begin{equation}\label{eq:ini}
u(t=0, x) \equiv u_0(x) \in C^{1,\alpha}_{loc}  (\mathbb{R}) , \quad \mbox{ where } \quad 0 \leq u_ 0 \leq 1
\end{equation}
for some $0\le \alpha < 1$. Following~\cite{Kim2020}, we introduce the notion of a solution for the equation~\eqref{eq:rd} with discontinuities.
\begin{defi}\label{def:sol1}
Assume that the initial condition $u_0 \in H^1_{loc}(\R)$. Define the functions:
$$\underline{f}  (s) := \min \{ f(s^-) , f (s^+)\}, \qquad \overline{f}  (s) := \max \{ f (s^-) , f(s^+) \}.$$
Let a function $u$ satisfy that, for any compact domain $D \subset \mathbb{R}$ and $T >0$,
\begin{equation*} 
u \in C ( [ 0,T); L^2 (D)) \cap H^1 ((0,T)\times D),
\end{equation*}
as well as $t \mapsto \int_\R e^{-|x|} u^2 (t,x) dx$ is continuous on $[0,T)$ and
\begin{equation*} 
\sup_{0< t < T} \int_\R e^{-|x|} ( |u(t,x)|^2 + |\partial_x u(t,x)|^2  ) dx < \infty.
\end{equation*}

Then $u$ is called a weak solution of~\eqref{eq:rd} if, for any nonnegative test function $\psi \in C_c^\infty (\mathbb{R}_+ \times \mathbb{R})$, one has
\begin{equation} \label{def:super1}
\int_\R u  (0,x) \psi (0,x) dx \geq \int \int_{\mathbb{R}_+ \times \mathbb{R}} (-u \partial_t \psi + \partial_x u \partial_x \psi -  \overline{f} (u) \psi ) dx dt,
\end{equation}
and
\begin{equation} \label{def:sub1}
\int_\R u (0,x) \psi (0,x) dx \leq \int \int_{\mathbb{R}_+ \times \mathbb{R}} (-u \partial_t \psi + \partial_x u \partial_x \psi - \underline{f} (u) \psi ) dx dt.
\end{equation}

If only the first (resp. second) inequality is satisfied, then we call $u$ a weak super-solution (resp. a weak sub-solution).
\end{defi}

\begin{rem}
	Since $u \in H^1((0,T)\times D)$, we have the time derivative $\partial_t u$ in $L^2((0,T)\times D)$ space. Then, the above two equations for the super- and sub-solutions in the definition, \eqref{def:super1}--\eqref{def:sub1}, can be respectively rewritten as follows:
	\begin{eqnarray*}
		\int \int_{\mathbb{R}_+ \times \mathbb{R}} (\partial_t u \, \psi + \partial_x u \partial_x \psi -  \overline{f} (u) \psi ) dx dt &\leq 0, \\ 
		\int \int_{\mathbb{R}_+ \times \mathbb{R}} (\partial_t u \, \psi + \partial_x u \partial_x \psi - \underline{f} (u) \psi ) dx dt &\geq 0.
	\end{eqnarray*}
\end{rem}
\begin{rem}
	In the Definition 5.1 of~\cite{Kim2020}, it is only assumed the $L^2((0,T);H^1(D))$ regularity for $u(t,x)$. But to obtain the uniqueness, we need an extra condition which is that $\partial_t u \in L^2((0,T)\times D)$; see indeed Theorem 5.2 of~\cite{Kim2020}. Hence we choose to assume in this paper that $u \in  H^1 ((0,T)\times D)$. 
\end{rem}
With this definition in hand, we recall the following well-posedness result:
\begin{theo}[\cite{Kim2020}]\label{thm:Kim2020}
The Cauchy problem~\eqref{eq:rd} with initial value $u_0$ which satisfies \eqref{eq:ini} admits a unique weak solution~$u$ in the sense of Definition~\ref{def:sol1}. Moreover, it satisfies that $u \in C^{ (1+\alpha)/2 , 1+\alpha}_{loc} (\mathbb{R}_+ \times \mathbb{R} ) \cap L^\infty (\mathbb{R}_+ \times \R)$.
\end{theo}
The proof of the existence part was based on a regularization technique, and approximating the discontinuous reaction~$f$ by a sequence of smooth functions. By the same regularization technique, we can prove further estimates and compactness properties that will be useful in studying the large time behavior of solutions; see Section~\ref{sec:prelim}. We further point out that the uniqueness of the solution follows from a weak comparison principle, also proved in~\cite{Kim2020}, and which we will use extensively here.\medskip

Next, we define some special solutions, namely traveling waves and terraces.
\begin{defi}\label{def:twsol}
	A $C^1$ function $\phi:\R\to\R$ is called a {\bf traveling wave solution} of \eqref{eq:rd} if there exists $c\in\R$ such that
\begin{equation*}
\phi''+c\phi'+f(\phi)=0
\end{equation*}
in the classical sense in the domain $\{z\in\R:\phi(z)\ne\theta_{2i},\  0 \leq i \leq I\}$. We refer to the constant~$c$ as the {\bf traveling wave speed}.

Furthermore, we say that the traveling wave monotonically connects the two steady states~$\theta_i$ and~$\theta_j$ with integers $0 \leq i < j \leq 2I$ if it is nonincreasing and
\begin{equation*}
\lim_{z \to-\infty}\phi(z)=\theta_j , \quad  \lim_{z\to+\infty}\phi(z)=\theta_i.
\end{equation*}
\end{defi}
One may check that, if $\phi$ is a traveling wave, then $u(t,x)=\phi(x-ct)$ is a solution of~\eqref{eq:rd} in the sense of the above Definition~\ref{def:sol1}. Therefore, by some slight abuse of language, we may refer to both $u(t,x)$ and $\phi(z)$ as traveling wave solutions whenever $u(t,x)=\phi(x-ct)$.

As we pointed out earlier, in the bistable case then there exists a traveling wave which monotonically connects the extremal stable steady states $0$ and $1$. However, if there are other stable steady states, then such a traveling wave may not exist, and the notion of a propagating terrace is considered instead as in~\cite{DGM}.
\begin{defi} \label{terracedef} A collection of traveling wave solutions $\{\phi_j:j=1,\cdots,J\}$ is called a {\bf propagating terrace} connecting $0$ and $1$ if each $\phi_j$ monotonically connects two steady states $\theta_{i_{j-1}}$ and $\theta_{i_{j}}$, and these limits and the wave speeds $c_j$ corresponding to $\phi_j$ satisfy
\begin{equation*}
0=\theta_{i_0}<\theta_{i_1}<\cdots<\theta_{i_J}=1\quad\text{and} \quad c_1\ge\cdots\ge c_J.
\end{equation*}
The steady states $\theta_{i_j}$'s are called the {\bf platforms} of the terrace.
\end{defi}
Before we recall our earlier result on the existence of a propagating terrace, for convenience we also introduce the following:
\begin{defi}\label{connectedcompact} The support of a function $\psi : \R \to \R$ is the set of points where $\psi$ is nonzero, that is
$$\text{spt} \, (\psi )= \{ z \in \mathbb{R} \, | \ \psi (z) \neq 0 \}. $$
Then a traveling wave solution $\phi$ is called: $(i)$ {\bf connected} if the support of $\phi'$ is connected; $(ii)$ {\bf compact} if the closure of the support of $\phi'$ is compact.
\end{defi}
The concept of a compact traveling wave is mostly meaningful in the discontinuous framework of \eqref{H1}-\eqref{H2}-\eqref{H3}. Indeed, if a traveling wave monotonically connects two stable steady states $\theta_i$ and $\theta_j$, where $f$ is discontinuous, then it is straightforward to check that this traveling wave must be compact. Indeed, if it is not then $\phi$ satisfies on the whole real line that $\theta_{i } < \phi < \theta_j$. Since $f$ is Lipschitz-continuous in this (open) interval, one can use standard elliptic estimates and find that $\phi '' (z) + c \phi ' (z) \to 0$ as $z \to \pm \infty$, which contradicts the fact that $f (\theta_i^+) < 0 < f (\theta_j^-)$. It turns out that, under the same assumptions, all the traveling waves involved in the propagating terrace indeed connect stable steady states.

We are now in a position to recall the main result on propagating terraces from our previous work:
\begin{theo}[\cite{GilettiKimKim_existence}]\label{th:exist} Let $f$ satisfy \eqref{H1}--\eqref{H2}--\eqref{H3}. There exists a propagating terrace connecting $0$ and $1$. It is unique in the sense that the set of platforms $\{ \theta_{i_j}\}$ is unique, and for each $j$ the traveling wave solution $\phi_j$ is unique up to translation.

Furthermore, the propagating terrace satisfies the following properties:
\begin{enumerate}[$(i)$]
\item the $\theta_{i_j}$'s are stable steady states, i.e. $i_j$ is an even integer for all $j=0,\cdots,J$;
\item the traveling waves $\phi_j$ are nonincreasing, connected and compact for all $j=1,\cdots,J$ in the sense of Definition~\ref{connectedcompact}.
\end{enumerate}
In particular, in the bistable case when $I=1$, then $J=1$ and the propagating terrace consists of a single traveling wave connecting $0$ and $1$.
\end{theo}
According to part $(ii)$ of Theorem~\ref{th:exist}, it is possible to shift traveling waves from the propagating terrace and to ``glue'' them by $C^1$-continuity to obtain another special solution of \eqref{eq:rd}. 
\begin{defi} \label{terracesolution}
Let $f$ satisfy \eqref{H1}--\eqref{H2}--\eqref{H3}, $\{\phi_j:j=1,\cdots,J\}$ be a propagating terrace, and $c_j$'s be the corresponding wave speeds. For given translation variables $\overrightarrow{\xi} = (\xi_1 , \cdots, \xi_J) \in\R^J$, the summation
\begin{equation}\label{Phi}
\Phi(t,x; \overrightarrow{\xi}):=\sum_{1 \leq j \leq J} (\phi_j (x - \xi_j - c_j t) - \theta_{i_{j-1}}),
\end{equation}
is called a {\bf terrace function}. 

If moreover $\Phi (\cdot,\cdot; \overrightarrow{\xi})$ solves \eqref{eq:rd} (possibly only for $t>T$ with $T>0$) then we call it a \textbf{terrace solution}.
\end{defi}
\begin{rem}\label{rem:solution}
Up to some shift, we may denote by $(0,\eta_j)$ the support of each~$\phi_j '$. Then it is straightforward to check that $\Phi (\cdot, \cdot; \overrightarrow{\xi})$ is a terrace solution for $t >0$ in the sense of Definition~\ref{terracesolution} if and only if
$$\xi_{j} \geq \xi_{j+1} + \eta_{j+1},$$
for any~$j$.
\end{rem}

We are now finally in a position to give our main findings. As announced earlier, our goal will be to show that these terrace solutions arise in the large time asymptotics of solutions of the Cauchy problem associated with \eqref{eq:rd}, for large classes of initial data. Thanks to the above properties, unlike in the case of a Lipschitz-continuous function $f$, here the attractiveness of the propagating terrace can be stated in a more straightforward way as the convergence of the solution to a terrace solution.

As a matter of fact, the large-time convergence to a traveling wave in the discontinuous framework is a new result, even in the bistable case. Therefore we first state our result in the simpler case when the propagating terrace contains a single traveling wave.
\begin{theo}[Convergence to a single traveling wave]\label{theo:stab_single} Let $f$ satisfy \eqref{H1}--\eqref{H2}--\eqref{H3}. Let also $\phi$ be a connected and compact traveling wave solution that monotonically connects two stable steady states $\theta_i< \theta_j$ and $c\in\R$ be its speed. Let also $u(t,x)$ be a solution of \eqref{eq:rd} supplemented with~\eqref{eq:ini}, and further assume that $u_0(x)$ is nonincreasing, and
\begin{equation*}
\lim_{x \to -\infty} u_0 (x) \in ( \theta_{j-1} , \theta_{j} ] , \qquad \lim_{x \to +\infty} u_0 (x) \in [\theta_i, \theta_{i+1}).
\end{equation*}
Then there exists $\xi_0 \in \R$ such that
\begin{equation*}
\|u(t,x)-\phi(x-\xi_0-ct)\|_{L^\infty (\R)} \to0\ \text{ as }\  t\to +\infty.
\end{equation*}
\end{theo}
Notice that, as explained previously, the assumption that $\phi$ is compact may be omitted since it is a consequence of the discontinuities of $f$, and the fact that its limiting steady states are stable. We included it anyway in our statement for convenience.

Theorem~\ref{theo:stab_single} is consistent with well-known results when the reaction function is smooth~\cite{FifeMcLeod77}. Yet we are also able to handle the case when the terrace contains an arbitrary number of platforms. Our most general convergence result reads as follows.
\begin{theo}[Convergence to a terrace solution]\label{theo:stab_terrace} Let $f$ satisfy \eqref{H1}--\eqref{H2}--\eqref{H3}, $\{\phi_j\}$ be the propagating terrace of \eqref{eq:rd} connecting $0$ and $1$, and $\{c_j\}$ be the corresponding wave speeds for $j=1\cdots,J$. Let also $u(t,x)$ be a solution of \eqref{eq:rd} supplemented with \eqref{eq:ini}, and further assume that $u_0(x)$ is nonincreasing, and
\begin{equation*}
\lim_{x\to -\infty} u_0 (x) \in (\theta_{2I-1}, 1],\qquad \lim_{x\to +\infty} u_0 (x) \in [0, \theta_1).
\end{equation*}
Then there exists $\overrightarrow{\xi}=(\xi_1,\cdots,\xi_J)\in\R^J$ such that
\begin{equation*}
\|u(t,x)-\Phi(t,x;\overrightarrow{\xi})\|_\infty\to0\ \text{ as }\  t\to\infty,
\end{equation*}
where $\Phi (t,x;\overrightarrow{\xi})$ is the terrace solution defined in~\eqref{Phi}.
\end{theo}
Again this is consistent with earlier results for smooth reaction functions~\cite{DGM,FifeMcLeod77,GM}, where some piecewise convergence of the solution toward each of the traveling waves of the propagating terrace was shown. However, as we already pointed out, the novelty of this result is the fact that the large-time dynamics is dictated by the terrace solution, which is a single special solution of~\eqref{eq:rd}.

\paragraph{Plan of the paper:} In Section~\ref{sec:prelim}, we start with some preliminaries which are some weak and strong comparison principles, as well as some parabolic estimates. Then in Section~\ref{sec:single}, we deal with the case when a single traveling wave connects directly $0$ and $1$ and prove our first main Theorem~\ref{theo:stab_single}. We again highlight the fact that this already includes the bistable case, which may be the most common in the applications.

Finally, in Section~\ref{sec:terrace}, we prove the more general result Theorem~\ref{theo:stab_terrace} when the terrace has an arbitrary number of waves. It may be worth pointing out that, as we will discuss there, this is actually a straightforward corollary of Theorem~\ref{theo:stab_single} wthen the speeds of the traveling waves are strictly ordered. However, additional difficulties arise when several waves of the terrace share the same speed, and Section~\ref{sec:terrace} will be mostly devoted to handling this situation.

\section{Preliminaries}\label{sec:prelim}

In this section, we give some useful results on solutions of \eqref{eq:rd} when $f$ may be discontinuous at the stable steady states. First we recall that a weak comparison principle is available~\cite{Kim2020}:
\begin{theo}[Weak comparison principle]\label{th:weak_comp}
Let $\underline{u}$ and $\overline{u}$ be respectively a sub and a super-solution of \eqref{eq:rd}. If $\underline{u} (0,\cdot) \leq \overline{u} (0,\cdot)$ then $\underline{u} (t,\cdot) \leq \overline{u} (t,\cdot)$ for all $t \geq 0$.
\end{theo}
It is rather clear that this comparison principle ensures that the solution of \eqref{eq:rd}-\eqref{eq:ini} is unique, and also that it remains bounded from above and below, respectively by $1$ and $0$. A key point of the proof of our convergence result will be to frame the solution between two terrace solutions, and this will rely on an extensive use of sub- and super-solutions.\medskip

Next we turn to some estimates on the solutions of \eqref{eq:rd}, which are to be understood in the weak sense of Definition~\ref{def:sol1}. These will be useful to get some compactness and pass to the limit as $t \to +\infty$. 

\begin{lem}\label{lemtn}
	Let $u(t,x)$ be a weak solution of \eqref{eq:rd}--\eqref{eq:ini}, and for any $t >1$ and $x_0 \in \mathbb{R}$, denote
$$Q_{T,R} (t_0,x_0) = (t_0,t_0+T) \times (x_0 - R, x_0 + R).$$
Then for any $p\geq 1$,we have
	\begin{equation}\label{estimation1}
	\| u (\cdot,\cdot) \|_{W^{1,2}_p (Q_{T,R} (t_0,x_0) )}^2  \le C_1 (p,T,R)< \infty ,
	\end{equation} 
and for any  $\beta \in (0,1)$, 
\begin{equation}\label{estimation2}
	\| u (\cdot,\cdot ) \|_{C^{(1+\beta)/2, 1+\beta} (Q_{T,R} (t_0,x_0 ))}  \le C_2 (\beta,T,R) < \infty,
	\end{equation}
where $C_1, C_2$ are positive constants which do not depend on $t_0$ and $x_0$.
\end{lem}
\begin{proof}
We apply the method used in the proof of Theorem~5.1 of \cite{Kim2020}, which is the existence result we recalled in Theorem~\ref{thm:Kim2020} above. The key idea is to approximate the discontinuous function $f$ by smooth functions $f^\varepsilon $. More precisely we fix $\varepsilon_0 >0$ and choose $f^\varepsilon (u)$ such that, for any $0 < \varepsilon < \varepsilon_0$:
	\begin{enumerate}
		\item $f^\varepsilon(\theta_{2i}) = \lim_{u \to \theta_{2i}^-} f(u)$ for all $0 \leq i \leq I$, and $|f^\varepsilon (u)| < M_1 +1$ for some $M_1 >0$ and any $u\ge 0$;
		\item $f^\varepsilon \downarrow f $ as $\varepsilon \to 0$ for all $u \in \R \setminus \{\theta_{2i}\,| \ 0\leq i \leq I \}$ and $f(u) \le f^\varepsilon(u) \le f(u) + \varepsilon$ for $u \notin  \cup_{0 \leq i \leq I} [\theta_{2i}, \theta_{2i}+\varepsilon]$;
		\item $f^\varepsilon(1+\delta) = 0$ for some $0 < \delta < \varepsilon$.
	\end{enumerate}	
	Then, we get a regularized problem, 
	\begin{equation} \label{continuous problem}
	\begin{cases}
	\partial_t u^\varepsilon = \partial_{xx} u^\varepsilon + f^\varepsilon(u^\varepsilon), & t >0 , \ \ x \in \R, \\
	u^\varepsilon(0,x) = u_0(x), & x \in \R ,
	\end{cases}
	\end{equation}
	which is well-posed in a classical sense, for instance by a standard monotone method~\cite{Sattinger}. 
We also have that $\overline{u} = 1+ \delta$ and $\underline{u} = 0$ are smooth super- and sub-solutions of \eqref{continuous problem}, so that
	\begin{equation*}
	0 \le u^\varepsilon (t,x) \le 1 +\delta, \quad \mbox{ for all } x \in \R,\  t>0.
	\end{equation*}
	Finally, thanks to (the proof of) Theorem 5.1 of \cite{Kim2020}, we know that $u^\varepsilon$ converges to $u$ in $C^{(1+\alpha)/2, 1+\alpha}_{loc}(\R_+\times \R)$.
	
Furthermore, by standard parabolic estimates (see for instance Theorem~5.2.5 in~\cite{KrylovSobolev}), we have for any $p \geq 1$, $t_0 >1$, $x_0 \in \mathbb{R}$ and $T,R>0$, that
$$\|u^\varepsilon\|_{W^{1,2}_p ( Q_{T,R} (t_0,x_0) ) } \leq C \| f^\varepsilon\|_{L^\infty (\mathbb{R})},$$
where $C$ is a positive constant depending only on~$T$, $R$ and~$p$. Then, by standard embeddings, we also get that
$$\| u^\varepsilon  \|_{C^{(1+\beta)/2, 1+\beta} (Q_{T,R} (t_0,x_0))}  \leq C \| f^\varepsilon\|_{L^\infty (\mathbb{R})} ,$$
for any $\beta \in (0,1)$, where $C$ possibly denotes another positive constant. Noting that $f^\varepsilon$ is bounded independently of $\varepsilon$, and passing to the limit as $\varepsilon \to 0$, we already get~\eqref{estimation1} and~\eqref{estimation2}.
	This completes the proof of Lemma~\ref{lemtn}.
\end{proof}

Lastly, our proofs will use some kind of sliding argument. However, such a sliding argument usually relies on a strong maximum principle, which is not valid in the discontinuous case. However, partial strong maximum type results still hold, in particular when considering a traveling wave solution. This is the subject of the following proposition. To avoid any ambiguity, let us clarify that, by an entire in time solution of \eqref{eq:rd}, we will simply mean a function $u$ such that $u (\cdot + \tau, \cdot)$ is a weak solution in the sense of Definition~\ref{def:sol1}, for any $\tau \in \R$.

\begin{prop}\label{prop:strong}
Let $f$ satisfy \eqref{H1}--\eqref{H2}--\eqref{H3}, and $\phi$ be a connected and compact traveling wave connecting $\theta_{2i}$ and $\theta_{2j}$ ($i<j$) in the sense of Definitions~\ref{def:twsol} and~\ref{connectedcompact}. Up to some shift, we denote
$$\text{spt} (\phi ') = (0,\eta).$$
Let also $u_\infty$ be an entire in time solution which we assume to be nonincreasing with respect to the space variable, and such that
$$u_\infty (t,x+ct) \geq \phi (x ),$$
for all $t \in \mathbb{R}$ and $x \in (0,\eta)$.
\begin{enumerate}[$(i)$]
 \item If $u_\infty (t_0 ,x_0+ct) = \phi (x_0)$ for some $t_0 \in \mathbb{R}$ and $x_0 \in (0,\eta)$, then 
$$u_\infty (t,x+ct) = \phi (x),$$
for all $t \in \mathbb{R}$ and $x \in [0,\eta]$.
\item If $u_\infty (t_0, \eta + ct_0) = \phi (\eta) = \theta_{2i}$ and $\partial_x u_\infty (t_0 , \eta + ct_0 ) \geq 0$ for some $t_0 \in \mathbb{R}$, then 
$$u_\infty (t,x+ct) = \phi (x),$$
for all $t \in \mathbb{R}$ and $x \in [0,\eta]$.
\end{enumerate}
\end{prop}
\begin{rem}
In this proposition we only consider situations where the contact point $x_0$ belongs to the extended support $(0,\eta]$ of the derivative of the traveling wave. Indeed, notice that for any $X>0$, the shift $\phi (\cdot - X)$ of the traveling wave is a nonincreasing in space and entire in time solution which satisfies that 
$$\phi (\cdot - X) > \phi \quad \text{ in $(0,X+\eta)$},$$
and $$\phi (\cdot - X) = \phi \quad \text{ in $(-\infty, 0] \cup [X +\eta, +\infty)$}.$$
Therefore a strong maximum principle cannot hold when $x_0 \not \in (0,\eta]$, which suggests that Proposition~\ref{prop:strong} is somewhat optimal.
\end{rem}
\begin{rem}
We point out that, under our multistable assumption, there may exist discontinuity points of $f$ between the two extremal steady states $\theta_{2i}$ and $\theta_{2j}$ of the traveling wave. This difficulty disappears in the bistable case, and therefore in such a situation one may simplify the proof. More specifically, it is enough to prove claim~\eqref{cla:new_newbis} below, and the recursive parts of the arguments can be skipped.
\end{rem}
\begin{proof}We prove both statements simultaneously. Without loss of generality, up to some time shift we assume that $t_0 = 0$. By assumption, we have that
$$u_\infty (0,x_0) = \phi (x_0) < \theta_{2j} ,$$
where either $x_0 \in (0, \eta)$ or $x_0 = \eta$ depending on which statement we consider. 

Then there exists an integer $i \leq \ell < j$ such that $u_\infty (0,x_0) \in [\theta_{2\ell} , \theta_{2\ell +2})$. For any $i < k < j$ we define $z_k \in (0,\eta)$ as the unique point such that
$$\phi (z_k ) = \theta_{2k},$$
and 
$$z_{j} = 0, \qquad z_{i} = \eta .$$ Then $x_0 \in (z_{\ell +1}, z_{\ell}]$, and we claim that
\begin{equation}\label{cla:new_newbis}
\forall t \in \R, \ \forall x \in [ z_{\ell +1} , z_{\ell } ] , \qquad u_\infty (t,x +ct) = \phi (x ).
\end{equation}
We also define, for any $t \in \R$,
$$x_{\theta_{2 \ell +2 },\infty} (t) := \inf \{  x \in \R \mid \ u_\infty (t,x+ct) < \theta_{2 \ell +2} \} \in (z_{\ell+1},\infty) .$$
Notice that $x_{\theta_{2\ell +2},\infty}$ is well-defined and larger than $z_{\ell+1}$ from the fact that $u_\infty$ is nonincreasing in space and $u_\infty (t,z_{\ell+1}+ ct) \geq \phi (z_{\ell+1}) =  \theta_{2 \ell+2}$ for all $t \in \mathbb{R}$. It is also upper semicontinuous in time.

We consider two subcases, and first we assume that $u_\infty (0,x_0) \in (\theta_{2\ell}, \theta_{2\ell +2})$, i.e. it is not a discontinuity point of the reaction. Due to the regularity of $f$, it follows from the standard strong maximum principle that $u_\infty (t,x+ct)$ and $\phi (x)$ coincide in any open and connected neighborhood of $(0,x_0)$ where both solutions remain in this interval. In particular, denoting
$$\omega_\ell :=  \{ (t,x) \in \R^2 \mid x_{\theta_{2\ell+2},\infty} (t) <  x <  z_{\ell} \},$$
then the connected component of its interior $\omega_\ell^\circ$ containing $(0,x_0)$ is such a subdomain. Indeed, it follows from the spatial monotonicity of $u_\infty$ and $\phi$, which together with the fact that $u_\infty (t,x+ct) \geq \phi (x)$, imply that $\theta_{2 \ell} < \phi (x), u (t,x+ct) < \theta_{2\ell +2}$ in $\omega_\ell$. Therefore, applying the usual strong maximum principle in this subdomain, and then due to the continuity of $u$ and $\phi$, one may infer that in fact $x_{\theta_{2\ell+2},\infty} (t) = z_{\ell +1}$ and $u_\infty (t,x+ct) = \phi (x )$ for all $t \in \R$ and $ z_{\ell +1}  \leq x \leq  z_{\ell } $. The claim is proved in that subcase.

The second subcase is when $u_\infty (0,x_0) =  \phi (x_0  ) = \theta_{2\ell }$. In the case of statement~$(i)$, we have $\ell > i$, and from the (strict) monotonicity of $\phi$ we must have
$$x_0 = z_{\ell }.$$
In the case of statement~$(ii)$, then by assumption we have that $x_0 = \eta$ and $\ell = i$, so that again $x_0 = z_{\ell}$. In either cases, we also assume that $u_\infty (t,x +ct) > \phi (x  )$ for all $z_{\ell +1} < x < z_{\ell}$, as otherwise we are back to the previous subcase. Then there exists $\delta >0$ such that $u_\infty (t,x) \in  (\theta_{2\ell }, \theta_{2\ell +2})$ for all $t \in [-\delta,0]$ and $ z_{\ell }- \delta< x  -ct < z_{\ell }$. Therefore it satisfies \eqref{eq:rd} in the classical sense in this subdomain, and we can use the standard Hopf lemma to infer that
$$\partial_x u_\infty (0, x_0) < \phi ' (x_0 ) \leq 0.$$
In the case of statement~$(i)$, we have that $x_0 \in (0,\eta)$ and this contradicts the fact that $u_\infty (0,\cdot) \geq \phi (0, \cdot )$ in $(0,\eta)$. In the case of statement~$(ii)$, it contradicts our assumption that $\partial_x u_\infty (0,x_0) \geq 0$. We conclude that \eqref{cla:new_newbis} holds true.

By \eqref{cla:new_newbis}, we have that $u_\infty (0, z_{\ell+1} ) = \phi (z_{\ell+1}) = \theta_{2\ell +2}$. Repeating the previous step, a straightfoward induction now implies that $u_\infty (t,x +ct) = \phi (x )$ for all $t\in \R$ and $x \in [0, z_{\ell }]$. If $\ell = i$, then $z_{\ell } = \eta$ and this concludes the proof. If $\ell > i$, then another induction will show that $u_\infty (t,x+ct) = \phi (x )$ for all $t \in \R$ and $x\in  [z_{\ell } , \eta]$. Indeed, let us prove that there exists~$\tilde{x}_0$ such that
$$u_\infty (0,\tilde{x}_0) = \phi (\tilde{x}_0 ) \in (\theta_{2\ell-2}, \theta_{2\ell }).$$
First, by \eqref{cla:new_newbis} we already know that
$$u_\infty (0,z_{\ell}) = \phi (z_{\ell} ) \ , \qquad \partial_x u_\infty (0,z_{\ell })  = \phi '  ( z_{\ell } ) < 0.$$
By the implicit function theorem, we get that $x_{\theta_{2 \ell },\infty}$ (defined in the same way as $x_{\theta_{2 \ell+2}}$ above) is $C^1$ on the time interval $[-\delta,0]$, and also that $u_\infty \in ( \theta_{2\ell -2}, \theta_{2\ell  } )$ in $ [-\delta,0] \times (x_{\theta_{2 \ell },\infty} (t),x_{\theta_{2 \ell },\infty} (t) +\delta)$  for some small enough $\delta >0$. We can again apply the standard Hopf lemma, and it follows that there must be $\tilde{x}_0 \in (x_{\theta_{2\ell },\infty} (0), x_{\theta_{2\ell },\infty} (0) + \delta)$ such that $u_\infty (0,\tilde{x}_0) = \phi (\tilde{x}_0 ) \in (\theta_{2 \ell -2}, \theta_{2 \ell})$. We can now repeat (the first subcase of) the proof of claim~\eqref{cla:new_newbis}, and by induction one eventually finds that $u_\infty (t,x+ct) = \phi (x )$ in~$\R \times [0,\eta]$. This concludes the proof of the proposition.
\end{proof}

\section{Convergence to a single traveling wave}\label{sec:single}

In this section we prove Theorem~\ref{theo:stab_single}. This in particular covers the bistable case when $I =1$. For convenience and without loss of generality, we assume that $\theta_i = 0$ and $\theta_j = 1$. We also recall that $\phi$ is a connected and compact traveling wave solution which connects $1$ and $0$, and $c \in \R$ is its speed. It is nonincreasing and, up to some shift, we further assume that
$$\text{spt} (\phi ' ) = ( 0,  \eta),$$
for some $\eta >0$, so that in particular $\phi (x) = 1$ for all $x \leq 0$, and $\phi (x) = 0$ for all $x \geq \eta$.

As far as the initial condition is concerned, we recall that
\begin{equation*}
u(t=0, x) \equiv u_0(x) \in C^{1,\alpha}_{loc}  (\mathbb{R}), \quad \mbox{ where } \quad 0 \leq u_ 0 \leq 1
\end{equation*}
for some $0\le \alpha < 1$. We further assume that
$$\lim_{x \to -\infty} u_0 (x) \in (\theta_{2I-1},1],$$
$$\lim_{x \to +\infty} u_0 (x) \in [0, \theta_{1}),$$
and that it is nonincreasing.

\subsection{A preliminary result}

We start by showing that the solution is identical to the steady states 0 and 1 outside of a bounded interval.
\begin{prop}[Connected and compact support]\label{th:support1}
	If $u(t,x)$ is a weak solution of \eqref{eq:rd} where $u_0$ satisfies \eqref{eq:ini}. Then there is some $T>0$ and two functions $x_0, x_1 : [T,+\infty) \to \mathbb{R}$ such that, for all $t \geq T$,
\begin{eqnarray}\label{ineq_xt}
\forall x > 0, & \qquad & u(t,x_0 (t) +x) = 0 , \quad u( t, x_1 (t) - x) = 1,   \vspace{3pt}\\
\forall x \in (x_1 (t), x_0 (t)),& \qquad &  0 <u (t,x) < 1.  \label{ineq_xt_bis}
\end{eqnarray}
\end{prop}
\begin{proof}
	First, we point out that $0 \leq u \le 1$ by the weak comparison principle for equation \eqref{eq:rd}. Also, for any $a>0$, $u(t,x-a)$ is another solution of \eqref{eq:rd} with initial value $u_0(x-a)$. Since~$u_0$ is a nonincreasing function, $u_0(x-a) \ge u_0(x)$ and, applying again the weak comparison principle,
	$$\forall t \ge 0, \quad u(t,x-a) \ge u(t,x).$$
In other words, $u(t,x)$ is a nonincreasing function of the variable $x$ for all $t \geq 0$.
	
	Now, for any $t \geq 0$, we define
$$x_0(t) := \sup\left\{ x\mid u(t,x) \in (0,1) \right\} \in \mathbb{R} \cup \{ \pm \infty \} , $$
$$x_1 (t) := \inf \left\{ x \mid u(t,x)  \in (0,1) \right\} \in \mathbb{R} \cup \{ \pm \infty \}.$$
In particular, $u(t,x_0 (t) +x) = 0$ and $u (t,x_1 (t) - x) = 1$  for all $t \geq 0$ and $x>0$. Due to the spatial monotonicity of $u$, it is also straightforward that $u (t ,\cdot) \in (0,1)$ in the interval $(x_1 (t), x_0 (t))$. It remains to check that $x_0$ and $x_1$ are real-valued functions.

Since $f (\theta_{1}) = 0$ and $f$ is Lipschitz-continuous on $(0,\theta_{1}]$, we have that
$$\forall s\in (0,\theta_{1}], \qquad  f(s) \geq - M (\theta_{1} - s),$$
for some $M>0$. It follows that the function
$$\underline{u} (t,x) :=  \max \{0, \theta_{1} - A e^{\sqrt{M} (x + 2 \sqrt{M} t)} \}, \quad \mbox{ with } \ A >0, $$
satisfies, whenever it is positive,
\begin{eqnarray*}
\partial_t \underline{u} - \partial_{xx} \underline{u} - f(\underline{u}) & \leq & 0.
\end{eqnarray*}
In particular, $\underline{u}$ is a weak sub-solution of \eqref{eq:rd}. Taking $A$ large enough so that $\underline{u} (0,\cdot) \leq u_0$, we infer that $x_0 (t) > -\infty$ for all $t \geq 0$. One can similarly check that $x_1 (t) < + \infty$ for all $t >0$.

Lastly, we check that $x_0 (t) < +\infty$ and $x_1 (t) > - \infty$. Due to our assumption on the initial data, these inequalities may actually be false at $t =0$. Yet we show here that they hold after some finite transient time.

Let us first find some time $T_0 \geq 0$ and $\overline{x}_0 \in \mathbb{R}$ such that $u(T_0 ,\overline{x}_0)=0$. Assume that such a point $(T_0 , \overline{x}_0)$ does not exist, or equivalently that $u$ is positive for all $t >0$ and $x \in \R$. Now define $\tilde{f}$ a Lipschitz-continuous function such that $\tilde{f} \geq f$ in $(0,1]$, as well as $\tilde{f} (-\varepsilon) = 0$ and $\tilde{f} (s) < 0$ for some $\varepsilon >0$ and all $s \in (-\varepsilon, \varepsilon)$. Since it is positive, the function $u$ also satisfies (in the weak sense)
$$\partial_t u \leq \partial_{xx} u + \tilde{f} (u),$$
By the comparison principle, one may infer that $\lim_{t\to +\infty} \lim_{x \to +\infty} u(t,x) = - \varepsilon$, a contradiction.

Thus there exist $T_0 \geq 0$ and $\overline{x}_0 \in \mathbb{R}$ such that $u(T_0, \overline{x}_0)=0$. By the spatial monotonicity of the solution, we also have that $u(T_0,x)=0$ for all $x \geq \overline{x}_0$, and in particular $x_0 (T_0) \leq \overline{x}_0 < +\infty$. It also follows that there exists some spatial shift $X$ such that
$$u (T_0 , x) \leq \phi (x - X),$$
for all $x \in \R$, where $\phi$ is the connected and compact traveling wave solution connecting $0$ and~$1$. Applying the comparison principle, we deduce that $x_0 (t) \leq X + ct < + \infty$ for all $t \geq T_0$. The argument for $x_1 (t) > -\infty$ is similar and we omit the details. The proposition is proved.
\end{proof}

\subsection{Proof of Theorem~\ref{theo:stab_single} }

We now turn to the proof of our first convergence result. The strategy is the same as in the more classical case of a Lipshitz-continuous reaction term. The first step is to trap the solution between two shifts of the traveling wave~$\phi$, and after that we will perform a sliding argument, which will be made possible by Proposition~\ref{prop:strong}.

First, it immediately follows from Proposition~\ref{th:support1} that the solution can be trapped between two shifts of the traveling wave. Together with the comparison principle, this leads to the following:
\begin{prop} \label{prop:trap1}
If $u(t,x)$ is a weak solution of \eqref{eq:rd}--\eqref{eq:ini}, there exists $T >0$ and $X >0$ such that
$$ \phi (x -c t + X) \leq u (t,x) \leq \phi (x -c t - X),$$
for all $t \geq T$ and $x \in \R$.
\end{prop}
In particular, any large-time limit of the solution in the moving frame with speed $c$ must also be trapped between two shifts of~$\phi$. As in the more standard case of a Lipschitz-continuous reaction, the existence of such large-time limit is ensured by some parabolic estimates, which we established here in Lemma~\ref{lemtn}.
\begin{prop}\label{prop:trap2}
	Let $u(t,x)$ be a weak solution of \eqref{eq:rd}-\eqref{eq:ini}, and $(t_n)_{n \in \mathbb{N}}$ be any time sequence with $t_n \to +\infty$ as $n \to +\infty$. Then, up to extraction of a subsequence, $u(t + t_n ,x+ c t_n)$ converges locally uniformly to some nonnegative and bounded function~$u_\infty$, which is also an entire in time solution of \eqref{eq:rd}. In addition, $\partial_x u(t+t_n,x+ct_n)$ converges locally uniformly to a nonpositive function $\partial_x u_\infty$.
\end{prop}
\begin{proof}

Applying Lemma~\ref{lemtn}, we get that up to extraction of a subsequence, $u(t+t_n,x+ct_n) \to u_\infty(t,x)$ in $C^{(1+\beta)/2, 1+\beta}_{loc}(\R_+ \times \R)$ for every $0<\beta < 1 $, with
	\begin{equation*}
	u_\infty \in C^{(1+\beta)/2, 1+\beta}_{loc}(\R_+ \times \R).
	\end{equation*}
	In particular, $u(t+t_n,x+ct_n)$ converges uniformly to $u_\infty$ as $n \to +\infty$ on any compact domain $(0,T) \times D$ with $T >0$ and $D \subset \mathbb{R}$. Since~$u$ is nonnegative and bounded by Proposition~\ref{th:support1}, the function~$u_\infty$ is also nonnegative and bounded. For the same sequence $t_n$, we also have that $\partial_x u (\cdot + t_n, \cdot + c t_n)$ converges locally uniformly to $\partial_x u_\infty$.
	
	The remaining part is to show that $u_\infty$ is a weak solution of~\eqref{eq:rd} in the sense of Definition~\ref{def:sol1}. By the above convergence, we clearly have that $u_\infty \in C ( [ 0,T); L^2 (D))$, and using again Lemma~\ref{lemtn}, we also get that $u_\infty \in H^1 ((0,T)\times D)$. Hence
	$$u_\infty \in C ( [ 0,T); L^2 (D)) \cap H^1 ((0,T)\times D).$$
Furthermore, we also have from the boundedness and the Hölder convergence that $t \mapsto \int e^{-|x|} u_\infty^2 (t,x) dx$ is continuous, and for any $T>0$,
	\begin{eqnarray}\nonumber 
	&&\sup_{0< t < T} \int_\R e^{-|x|}( |u_\infty|^2 + |\partial_x u_\infty|^2 ) dx \\ \nonumber &=& \sup_{0<t<T} \lim_{t_n \to \infty} \Big(\|e^{-{|x|\over 2}} u(t+t_n,x) \|_{L^2(\R)}^2 + \|e^{-{|x|\over 2}} \partial_x u(t+t_n,x) \|_{L^2(\R)}^2 \Big) \\ \nonumber &\le&\sup_{t>0} \Big(\|e^{-{|x|\over 2}} u(t,x) \|_{L^2(\R)}^2 + \|e^{-{|x|\over 2}} \partial_x u(t,x) \|_{L^2(\R)}^2 \Big)<\infty.
	\end{eqnarray}
Here in the last inequality we used again Lemma~\ref{lemtn}.
	
	Finally, we need to show that $u_\infty$ satisfy the inequalities \eqref{def:super1} and \eqref{def:sub1}. We will only show the super-solution case, \eqref{def:super1}. For any nonnegative test function $\psi \in C_c^\infty (\mathbb{R}_+ \times \mathbb{R})$ and the sequence $t_n$ which makes $u(t+t_n,x+ct_n) \to u_\infty(t,x)$, we have 
	\begin{equation*} 
	I_n := \int \int_{\mathbb{R}_+ \times \mathbb{R}} - u(t+t_n,x+ct_n) \, \partial_t \psi + \partial_x u(t+t_n,x+ct_n) \partial_x \psi -  \overline{f} (u) \psi \, dx dt - \int_\R u(t_n,x+ct_n)\psi dx \leq 0. 
	\end{equation*}
	Now we consider
	\begin{equation*} 
	I := \int \int_{\mathbb{R}_+ \times \mathbb{R}} - u_\infty(t,x) \, \partial_t \psi + \partial_x u_\infty(t,x) \partial_x \psi -  \overline{f} (u_\infty) \psi \, dx dt - \int_\R u_\infty(0,x)\psi dx.
	\end{equation*}
	Our goal is to show that $I \le 0$ and it can be obtained by first observing that 
	$$ I = I-I_n + I_n  \le I-I_n.$$
	Since the above inequality is satisfied for all $n$, we only need to show that $\lim_{n\to\infty} I_n-I = 0$. 
	
	Let $Q_{T,R} = (0,T) \times (-R,R)$ include the compact support of $\psi$. For convenience, denote $u_n(t,x) = u(t+t_n,x+ct_n)$. Then for all $n$,
	\begin{eqnarray*}
		I-I_n  &=& \int_0^T \int_{-R}^R - (u_\infty-u_n) \, \partial_t \psi + (\partial_x u_\infty - \partial_x u_n) \, \partial_x \psi -  (\overline{f} (u_\infty)-\overline{f} (u_n)) \psi \, dx dt \\ &&\quad - \int_\R (u_\infty(0,x)-u_n(0,x)) \psi(0,x) dx \\
		&\le& C \left(  \sup_{Q_{T,R}} |u_\infty-u_n| + \sup_{Q_{T,R}} |\partial_x u_\infty-\partial_x u_n| \right) - \int_0^T \int_{-R}^R(\overline{f} (u_\infty)-\overline{f} (u_n)) \psi \, dx dt ,
	\end{eqnarray*}
for some $C>0$. From the convergence of $u_n$ to $u$ in $C^{(1+\beta)/2, 1+\beta}_{loc}$, 
$$\lim_{n\to\infty} \Big[\sup_{Q_{T,R}} |u_\infty-u_n| + \sup_{Q_{T,R}} |\partial_x u_\infty-\partial_x u_n|\Big] = 0.$$
Next, take $\varepsilon >0$ arbitrarily small and denote by $\tilde{Q}_\varepsilon \subset Q_{T,R}$ the subdomain where $u_\infty \in \bigcup_i (\theta_{2i}+\varepsilon,\theta_{2i+2}- \varepsilon) $; i.e. it is away from the discontinuity points of $f$. Thanks to the definition of $\overline{f}$ and the Lipschitz-continuity of $f$ outside of its discontinuity points, we get that
$$\sup_{\tilde{Q}_\varepsilon}|\overline{f} (u_\infty)-\overline{f} (u_n) | \leq K \sup_{Q_{T,R}} | u_\infty - u_n|,$$
for some $K>0$.

Now consider $\widehat{Q} \subset Q_{T,R}$ the subdomain where $u_\infty \in \{ \theta_{2i} \, | \ 0 \leq i \leq I \}$. By the lower semi-continuity of $\overline{f}$, we have
$$\lim_{n \to \infty} \sup_{\widehat{Q}}\overline{f} (u_\infty) - \overline{f} (u_n) \leq 0.$$
Finally, noticing that the Lebesgue measure of $Q_{T,R} \setminus (\widehat{Q} \cup \tilde{Q}_\varepsilon)$ goes to 0 as $\varepsilon \to 0$, we conclude that
$$\limsup_{n \to \infty} \int_0^T \int_{-R}^R(\overline{f} (u_\infty)-\overline{f} (u_n)) \psi \, dx dt \leq 0.$$
Thus, we have $\lim_{n\to\infty} I-I_n \le 0$ and $u_\infty$ satisfies \eqref{def:super1}.

	For the sub-solution case, we can obtain $I \ge I-I_n \ge 0$ similarly and we can deduce that~$u_\infty$ also satisfies~\eqref{def:sub1}. Thus, $u_\infty$ is a weak solution in the sense of Definition~\ref{def:sol1} and the proof is completed.
\end{proof}

We can now turn to the second part of the proof, which is to perform some sliding argument to find that the large-time limit of the solution in the moving frame with speed $c$ is unique and coincides with a shift of the traveling wave. To set up this argument, we first define
$$\Omega_{c} (u) := \{ u_\infty \, | \ \ \exists t_n \to +\infty, \ u(t+ t_n, x + c t_n ) \to u_\infty (t,x) \mbox{ loc. unif. as } n \to +\infty \},$$
the set of limits of the solution in the moving frame with speed $c$. Proposition~\ref{prop:trap2} ensures that this set is not empty. It also ensures that
$$X_- := \inf \{ X \, | \ \  \forall u_\infty \in \Omega_{c} (u),  \  \forall (t,x) \in \R^2, \ \phi (x-c t + X) \leq u_\infty (t,x) \} ,$$
and
$$X_+ := \inf \{ X \, | \ \  \forall u_\infty \in \Omega_{c} (u),  \  \forall (t,x) \in \R^2, \ \phi (x-c t - X) \geq u_\infty (t,x) \} ,$$
are well-defined real numbers. Observe that, from a simple continuity argument,
\begin{equation}\label{uinfty_trapped}
\forall u_\infty \in \Omega_{c} (u), \ \forall (t,x) \in \R^2, \ \phi (x-c t + X_-) \leq u_\infty (t,x) \leq \phi (x-ct - X_+).
\end{equation}
Our goal now is to prove that $X_- = -X_+$, so that
$$\Omega_{c} (u) = \{ (t,x) \mapsto \phi (x- c t + X_-) \}.$$

We start with the following lemma:
\begin{lem}\label{lem:stab_new1}
Assume that there exist $X \in \R$ and $u_\infty \in \Omega_c (u)$ such that
$$ u_\infty ( 0,x) \geq \phi (x + X),$$
for all $x \in \R$.

Then any $\tilde{u}_\infty \in \Omega_c (u)$ also satisfies
$$\tilde{u}_\infty (t,x) \geq \phi (x   - ct + X),$$
for all $(t,x)\in \R^2$.
\end{lem}
\begin{proof}
We fix any $\varepsilon >0$. We claim that there exists $T >0$ such that
\begin{equation}\label{lem_newclaim}
u (T, x + c T ) \geq \phi (x  +X +\varepsilon).
\end{equation}
First, by the comparison principle, we have that $u_\infty (t,x) \geq \phi (x - ct + X)$ for all $t \geq 0$. Due to the strict monotonicity of $\phi$ in $(0, \eta)$,
$$u_\infty (t,x) > \phi (x - ct +X  + \varepsilon),$$
for all $t \geq 0$ and $x \in [ ct  - X - \varepsilon/2,\eta+ct - X - \varepsilon ]$.

On the other hand, by the definition of $\Omega_c (u)$, there exists a time sequence $t_n$ such that
$$u (t+ t_n, x + ct_n) \to u_\infty (t,x),$$
as $n \to +\infty$, where the convergence is to be understood in the locally uniform sense. In particular, for any $n$ large enough we have that
$$u(t + t_n, x + ct + c t_n) \geq \phi (  x  +X + \varepsilon),$$
for all $t \in [0,1]$ and $x \in  [  - X  - \varepsilon/2 , \eta - X -  \varepsilon]$. Moroever, for any $t \in [0,1]$ and $x \geq \eta - X - \varepsilon$, then
$$u (t+ t_n , x + ct + ct_n) \geq 0 = \phi (x + X + \varepsilon).$$
Lastly, we check that there exists $\tau \in [0,1]$ such that
$$x_1 (t_n+  \tau ) \geq c \tau + c t_n  - X - \frac{\varepsilon}{2}.$$
If true, then we conclude that $u  (\tau+ t_n  , x +c \tau + c t_n) \geq \phi (x + X + \varepsilon)$ on the whole real line, and claim~\eqref{lem_newclaim} holds true with $T = \tau+t_n $. Otherwise, we would have
$$x_1 (t + t_n) <  c t + c t_n - X - \frac{\varepsilon}{2},$$
for all $t \in [0,1]$. Then, recalling that $u (t + t_n,  c t + ct_n - X  ) \to  u_\infty (t, ct - X) \geq \phi (0 )= 1$ locally uniformly as $n \to +\infty$, we get up to extracting another subsequence that
$$1 - \frac{1}{n} < u (t + t_n, x + ct + ct_n )< 1,$$
for all $t \in [0,1]$ and $x \in [ - X - \varepsilon/2,  - X]$. Define $I = [ - X - \varepsilon/2,  - X]$ which is a compact interval.  In particular, $\tilde{u}_n (t,x) := u (t + t_n, x + ct + ct_n) $ solves
$$\partial_t  \tilde{u}_n = \partial_{xx} \tilde{u}_n + c \partial_x \tilde{u}_n + \tilde{f} (\tilde{u}_n)$$
on $[0,1] \times I$, where $\tilde{f}$ is any Lipschitz-continuous which coincides with $f$ on $(1 -1/n , 1)$. Integrating the equation on $I$, we get
$$\partial_t \int_I \tilde{u}_n dx =  \partial_x \tilde{u}_n \mid_{I} + c \tilde{u}_n \mid_I + \int_I \tilde{f}(\tilde{u}_n) dx.
$$
Recall that, from the proof of Proposition~\ref{prop:trap2}, we also have that $\partial_x u(t+t_n, x+ct+ct_n)$ converges locally uniformly as $n \to \infty$ to $\partial_x u_\infty (t,x+ct)$, which is 0 in $[0,1] \times I$. Moreover, there exists $f^* >0$ such that $\tilde{f} (s)  = f(s) > f^*$ for any $s \in (1 - \frac{1}{n}, 1)$. Hence
$$\liminf_{n \to \infty} \partial_t \int_I \tilde{u}_n  dx \ge \frac{\varepsilon}{2} f^*,$$
and
$$\liminf_{n \to \infty} \int_I \tilde{u}_n (t=1) \ge \frac{\varepsilon}{2} \left( f^* + 1  \right).$$
This leads to a contradiction as such solution should exceed 1 at some point when $t = 1$. 
Thus, we have that there exists $\tau \in [0,1]$ such that
$$x_1 (t_n+  \tau ) \geq c \tau + c t_n  - X - \frac{\varepsilon}{2}.$$
We conclude that \eqref{lem_newclaim} holds for some $T>0$, and that $u(t, x+ ct) \geq \phi (x + X+ \varepsilon)$ for all $t \geq T$ by the weak comparison principle. It follows that $\tilde{u}_\infty (t,x) \geq \phi (x-ct + X + \varepsilon)$ in~$\R^2$ for any~$\tilde{u}_\infty \in \Omega_c (u)$ and, since $\varepsilon$ can be chosen arbitrarily small, we reach the wanted conclusion.
\end{proof}
The following result is an immediate consequence of Lemma~\ref{lem:stab_new1}:
\begin{prop}\label{prop:stab_11}
Assume that $(t,x) \mapsto \phi (x-c t + X_1)$ and $(t,x) \mapsto \phi (x-c t + X_2)$ belong to~$\Omega_{c} (u)$. Then $X_1= X_2$.
\end{prop}
Finally, we prove that:
\begin{prop}\label{prop:stab_12}
The functions $(t,x) \mapsto \phi (x-c t + X_-)$ and $(t,x) \mapsto \phi (x - ct - X_+)$ belong to~$\Omega_{c} (u)$.
\end{prop}

Putting this together with Proposition~\ref{prop:stab_11}, we get that $X_- = - X_+$. By the definition of $X_-$ and $X_+$ it clearly follows that $\Omega_c (u) = \{ (t,x) \mapsto \phi (x-ct + X_-) \}$ and the stability Theorem~\ref{theo:stab_single} is proved.
\begin{proof}[Proof of Proposition~\ref{prop:stab_12}]
We only prove that $(t,x) \mapsto \phi (x-ct + X_-)$ belongs to $\Omega_c (u)$. The case of the other shift follows by a symmetry argument, letting $v (t,x) = 1- u(t,-x)$ which solves another reaction-diffusion equation with the same set of assumptions.

First recall that $spt (\phi ' ) = (0,\eta)$. In particular $ \phi (x) = 1$ for all $x \leq 0$ and $\phi (x)= 0$ for all $x \geq \eta$. Due to the definition of $X_-$, for any $\varepsilon >0$, there exists $u_\infty \in \Omega_{c} (u)$ such that
\begin{equation}\label{eq:lem_limit1}
\inf_{(t,x) \in \R^2} u_\infty (t,x) - \phi (x - c t + X_- - \varepsilon ) < 0.
\end{equation}
We further claim that
\begin{equation}\label{eq:lem_limit2}
\inf_{t \in \R} \inf_{x - ct + X_- - \varepsilon \in (0,\eta)}  u_\infty (t,x) - \phi (x-ct + X_- - \varepsilon ) < 0.
\end{equation}
Otherwise, we would have that $u_\infty (t,ct - X_- + \varepsilon) = 1$ for all $t \in \R$, hence $u_\infty (t,x) = 1 = \phi (x-ct + X_- - \varepsilon)$ for all $t \in \R$ and $x \leq ct - X_- + \varepsilon$. Moreover, $u_\infty (t,x) \geq 0 = \phi (x -ct + X_- - \varepsilon)$ for all $t \in \R$ and $x \geq ct - X_- + \varepsilon+\eta$. In other words, \eqref{eq:lem_limit2} must hold not to contradict \eqref{eq:lem_limit1}.

Now, by \eqref{eq:lem_limit2} and for any $n \in \mathbb{N}^*$, there exist $u_{\infty,n} \in \Omega_{c} (u)$ and some $(s_n, y_n) \in \mathbb{R}^2$ such that
$$u_{\infty, n} (s_n,y_n) < \phi \left(y_n - c s_n + X_- - \frac{1}{n} \right) , $$
and
$$ c s_n - X_- + \frac{1}{n} \leq y_n \leq \eta+c s_n - X_- + \frac{1}{n}.$$
In particular,
$$ \sup_{n \in \mathbb{N}}  | y_n - c s_n | < + \infty .$$
Then, since the functions $u_{\infty, n}$ belong to $\Omega_{c} (u)$, one can also find a sequence $t_n$ such that $t_n + s_n \to +\infty$ and
$$u (t_n + s_n , c t_n + c s_n + (y_n - c s_n) ) < \phi \left(y_n - c s_n + X_- - \frac{1}{n} \right)  .$$
Passing to the limit as $n \to +\infty$, and possibly up to extraction of another subsequence, we find some new $u_\infty \in \Omega_{c} (u)$ and $y_\infty = \lim_{n \to +\infty} y_n - c s_n \in [ -X_-,  \eta- X_- ]$ such that
$$u_\infty (0,  y_\infty)  = \phi (y_\infty + X_- ).$$
On the other hand, recall that by definition of $X_-$ we also have
$$u_\infty (t,x) \geq \phi (x - c t + X_-),$$
for all $(t,x) \in \mathbb{R}^2$, and in particular
\begin{equation}\label{x1inftybound}
x_{0,\infty} (t) \geq c t - X_- + \eta, \qquad x_{1,\infty} (t) \geq c t -X_- , 
\end{equation}
for all $t \in \R$, where the functions $x_{0,\infty}$, $x_{1,\infty}$ are such that \eqref{ineq_xt}-\eqref{ineq_xt_bis} hold with $u_\infty$ instead of~$u$. Notice that such functions are well defined thanks to~\eqref{uinfty_trapped}.\medskip

\textit{Case 1.}
Let us first assume that $u_\infty (0,y_\infty) = \phi (y_\infty + X_-) \in (0,1)$. Then by statement~$(i)$ of Proposition~\ref{prop:strong}, we get that 
$$u_\infty (t,x+ct) = \phi ( x + X_-),$$
for all $t \in \mathbb{R}$ and $x \in [-X_-,\eta - X_-]$. Due to the spatial monotonicity of $u_\infty$ and the fact that $0 \leq u_\infty \leq 1$, we even have that $$u_\infty (t, x ) \equiv \phi  (x  -ct + X_-),$$
in $\mathbb{R}^2$. In particular $(t,x) \mapsto \phi (x-ct + X_-)$ belongs to $\Omega_c (u)$.\medskip

\textit{Case 2.} Next assume that $u_\infty (0, y_\infty)   = \phi (y_\infty + X_-) = 0$, so that $y_\infty = \eta - X_-$. By the nonnegativity of $u_\infty$, we must also have that $\partial_x u_\infty (0,y_\infty)  \geq 0$. Then by statement~$(ii)$ of Proposition~\ref{prop:strong}, and as in the previous case, we conclude that $u_\infty (t,x) \equiv \phi (x - ct + X_-)$ belongs to $\Omega_c (u)$.\medskip

\textit{Case 3.} The arguments of the previous two cases apply to any time shift of $u_\infty$. Therefore, it only remains to consider the case when $u_\infty (t,x) > \phi (x  -ct + X_-)$ for all $t \in \R$ and $ - X_-  < x - ct \leq \eta - X_-$, and $u_\infty (0,  - X_-) = \phi ( 0) = 1$. As we will see, this last case actually leads to a contradiction.

We consider two subcases. The first subcase is the one when	
$$\exists \delta >0, \ \forall t \in [-\delta, 0], \quad x_{1,\infty} (t) = c t -  X_-  .$$
In that case we can actually apply the standard Hopf lemma on $$\{ (t,x) \, | \ t \in [-\delta,0] \mbox{ and } x \in [ct - X_-, ct - X_- +\epsilon]\},$$ for some $\epsilon >0$, and conclude that
$$\partial_x u_\infty (0, - X_-) > \phi '(0) = 0.$$
Since $u_\infty$ is $C^1$ in the spatial variable and $u_\infty \leq 1 = u_\infty (0,-X_-)$, this is a contradiction.

Now we turn to the second subcase and (recall~\eqref{x1inftybound}) assume that there is $t_1 < 0$ such that
$$x_{1,\infty} (t_1) > c t_1 - X_- .$$
Recalling here that $u_\infty (t_1,x) > \phi (x-ct_1 + X_-)$ for all $x \in (ct_1 - X_- , \eta+ct_1 - X_-]$, it is then straightforward that, for $\varepsilon >0$ small enough,
$$u_\infty (t_1, x) \geq \phi (x - ct_1 + X_- -\varepsilon ),$$
for all $x \in \R$. By the comparison principle, we also get that $u_\infty  (0, x) \geq \phi (x + X_- - \varepsilon)$. Applying Lemma~\ref{lem:stab_new1}, it follows that $\tilde{u}_\infty (t,x) \geq \phi (x -ct +X _- - \varepsilon)$ for all $\tilde{u}_\infty \in \Omega_c (u)$ and $(t,x) \in \R^2$. This in turn contradicts the definition of $X_-$.\medskip

Finally, we conclude that the only possibilities are those considered in Cases~1 and~2, hence $u_\infty (t,x)$ coincide with $\phi (x - ct + X_-)$ and $(t,x) \mapsto \phi (x - ct + X_-)$ belongs to $\Omega_c (u)$. Proposition~\ref{prop:stab_12} is proved.
\end{proof}

\section{Convergence to a terrace solution}\label{sec:terrace}

In this section we turn to the multistable case where the terrace solution may consist of two waves or more, and the proof of Theorem~\ref{theo:stab_terrace}. Here we denote by $\phi_j, \ j = 1,\cdots,J$ the traveling waves constituting the terrace, connecting respectively $\theta_{i_{j-1}} = p_{j-1}$ and $\theta_{i_j} = p_{j}$ with speed $c_j$ such that
\begin{equation*}
0 = p_0 < p_1 < \cdots < p_J = 1 \quad \mbox{and} \quad c_1 \ge c_2 \ge \cdots \ge c_J.
\end{equation*}
For convenience we also define $c_0 = +\infty$ and $c_{J+1} = -\infty$. Since we already addressed convergence to a single traveling wave in the previous section, here we assume $J \ge 2$.

We denote the terrace solution
\begin{equation*}
\Phi(t,x; \overrightarrow{\xi} \ ):=\sum_{1 \leq j \leq J} (\phi_j (x - \xi_j - c_j t) - p_{j-1})
\end{equation*}
for some shift $\overrightarrow{\xi} = (\xi_1, \cdots, \xi_J)$. Recall Remark~\ref{rem:solution} and note that $\overrightarrow{\xi}$ should be taken at least satisfying $\xi_j \ge \xi_{j+1} + \eta_{j+1}$ for $j=1,\cdots,J-1$ if $\spt(\phi_j^\prime) = (0,\eta_j)$ for all $j$.

Moreover, as we will see below the main difficulty arises when several waves of the terrace solution have the same speed. Therefore, when $c =c_{i+1} = \cdots = c_{i+k}$ for some integers $0 \leq i < i+ k \leq J$, we also introduce the notation
\begin{equation}\label{partialPhi}
\Phi^c(t,x; \overrightarrow{\xi}_c \ ):=p_i + \sum_{j=i+1}^{i+k} (\phi_j (x - \xi_j - c t) - p_{j-1}),
\end{equation}
which we refer to as a partial terrace solution, constituting of all traveling waves moving with the same speed~$c$. Notice that $\Phi^c$ is also a terrace solution connecting $p_i$ and $p_{i+k}$, and that~$\Phi$ and~$\Phi^c$ coincide in the moving frame with speed~$c$. Equivalently,
\begin{equation}\label{partialPhiback}
\Phi (t, x ; \overrightarrow{\xi} \ ) = \sum_{ c \in \{c_1,\cdots, c_J\} } (\Phi^c (t,x; \overrightarrow{\xi}_c ) - p_i ) ,
\end{equation}
where for each $c$, the integer~$i+1$ denotes the index of the lowest traveling wave of the partial terrace $\Phi^c$. \medskip

Our goal in this section is to show the following asymptotics theorem:
\begin{theo}\label{asymptotics:global} 
	Let $u(t,x)$ be the solution of \eqref{eq:rd}-\eqref{eq:ini} where $u_0$ is nonincreasing and
$$\lim_{x \to -\infty} u_0 (x) \in (\theta_{2I-1},1], \qquad \lim_{x \to +\infty} u_0 (x) \in [0, \theta_{1}).$$
	\begin{enumerate}[(i)]
		\item If $c \in (c_{j+1}, c_j)$ for some $j=0, \cdots, J$, then we have $u(t+t_n,x+ct_n) \to p_i$ locally uniformly.
		\item If $c = c_j$ for some $j=1,\cdots,J$, then we have $u(t+t_n,x+ct_n) \to \Phi^c(t,x; \overrightarrow{\xi}_c \ )$ locally uniformly for some shift $\overrightarrow{\xi}_c = (\xi_{i+1}, \cdots, \xi_{i+k})$, where we recall that $\Phi^c$ denotes the partial terrace solution constituting of all waves moving with speed~$c$.
	\end{enumerate}
\end{theo}
Notice that Theorem~\ref{theo:stab_terrace}, that is the uniform convergence to a terrace solution, immediately follows from Theorem~\ref{asymptotics:global}, together with~\eqref{partialPhiback} and the spatial monotonicity of~$u_0$, which is inherited by the solution in vertue of the weak maximum principle.\medskip

In order to establish Theorem~\ref{asymptotics:global}, first of all we adapt Propositions~\ref{th:support1} and~\ref{prop:trap1} to the multistable case.
\begin{prop}\label{prop:trap3}
	Let $u(t,x)$ be a weak solution of \eqref{eq:rd}-\eqref{eq:ini} with nonincreasing initial value~$u_0$, and define $2J$ functions $x_{j}^l, \ x_{j}^u \ : \ [0,+\infty) \to \R$ for $j=1,\cdots,J$ as 
	$$ x_{j}^l(t) := \sup\left\{ x \in \R\cup \left\{ \pm\infty \right\} \mid u(x,t) \in (p_{j-1}, p_j) \right\}, $$
	$$ x_{j}^u(t) := \inf\left\{ x \in \R\cup \left\{ \pm\infty \right\} \mid u(x,t) \in (p_{j-1}, p_j) \right\}. $$
	Then, there is a constant $T>0$ such that, for all $t \ge T$, 
	\begin{equation*}
		-\infty < x_{J}^u(t) \le x_{J}^l(t) \le x_{J-1}^u(t) \le \cdots \le x_{1}^l(t) < +\infty.
	\end{equation*}
In particular, there exist $\overrightarrow{\xi}_{sup}$ and $\overrightarrow{\xi}_{sub}$ such that 
	$$ \Phi(t,x;\overrightarrow{\xi}_{sub}) \le u(t,x) \le \Phi(t,x;\overrightarrow{\xi}_{sup}), $$
	for all $t\ge T$ and $x \in \R$.
\end{prop}
Since the proof is the same as that of Propositions~\ref{th:support1} and~\ref{prop:trap1}, here we omit it. Next we point out that Proposition~\ref{prop:trap2} is still valid regardless of the number of waves in the propagating terrace. Thus, from any sequence $n \to +\infty$ such that $t_n \to +\infty$ and $x_n \in \mathbb{R}$, we can extract a subsequence such that $u(t+t_n,x+ct_n)$ converges locally uniformly to some entire in time solution~$u_\infty (t,x)$. Furthermore, by Proposition~\ref{prop:trap3} we have
\begin{equation} \label{inequality1}
	\lim_{n\to\infty} \Phi(t+t_n,x+ct_n;\overrightarrow{\xi}_{sub}) \le u_\infty(t,x) \le \lim_{n\to\infty} \Phi(t+t_n,x+ct_n;\overrightarrow{\xi}_{sup}).
\end{equation}
If $c \in (c_{j+1}, c_j)$ for some $j$, then
$$ \lim_{n\to\infty} \Phi(t+t_n,x+ct_n;\overrightarrow{\xi}_{sub}) = \lim_{n\to\infty} \Phi(t+t_n,x+ct_n;\overrightarrow{\xi}_{sup}) = p_j.$$
Thus statement~$(i)$ of Theorem~\ref{asymptotics:global} already follows from~\eqref{inequality1}.

In the remainder of this section, we will show the second assertion of Theorem~\ref{asymptotics:global}, and assume that 
$$c_i < c = c_{i+1} = \cdots = c_{i+k} < c_{i+k+1},$$ for some integers $i \geq 0$ and $k \geq 1$. Notice that 
$$u \left( t , \frac{c + c_i}{2} t  \right) = p_i, \qquad u \left( t, \frac{c + c_{i+k+1}}{2} t \right) = p_{i+k+1},$$
$$\partial_x u \left( t , \frac{c + c_i}{2} t  \right) = \partial_x u \left( t, \frac{c + c_{i+k+1}}{2} t \right) = 0,$$
for all $t$ large enough. In particular, defining
$$\tilde{u} (t,x ) =  \left\{ 
\begin{array}{ll}
\displaystyle \frac{u (t,x) - p_i}{p_{i+k+1} - p_i} &\displaystyle  \text{ if } \,  \frac{c + c_{i+k+1}}{2} t \leq x \leq \frac{c + c_{i}}{2} t ,\vspace{3pt} \\
\displaystyle 0  & \displaystyle \text{ if } \, x >  \frac{c + c_{i}}{2} t , \vspace{3pt}\\
\displaystyle 1  & \displaystyle \text{ if } \, x <  \frac{c + c_{i+k+1}}{2} t ,
\end{array}
\right.
$$
we have that $0 \leq \tilde{u} \leq 1$ solves an equation of the type~\eqref{eq:rd}, which is still multistable in the sense of \eqref{H1}-\eqref{H2}-\eqref{H3}, and which by construction admits a terrace solution whose waves move with the same speed. Furthermore, the convergence of $\tilde{u}$ to a terrace solution is clearly equivalent to that of $u$ in the moving frame with speed~$c$. In other words, dropping the tilde for convenience, without loss of generality we can assume that $i=0$ and $k =J$, so that
$$\Phi (\cdot,\cdot ; \overrightarrow{\xi}) \equiv \Phi^c (\cdot,\cdot;\overrightarrow{\xi}_c ),$$
where $\Phi^c$ was defined in~\eqref{partialPhi}. Then our goal is still to prove that $u(t,\cdot + ct)$ converges locally uniformly to a terrace solution~$\Phi^c$ as $t \to +\infty$. In the case when $\Phi^c$ contains a single traveling wave, then we are back to the situation tackled in Theorem~\ref{theo:stab_single}. In particular, when the waves of the terrace have strictly ordered speeds, i.e.
$$c_1 < c_2 < \cdots < c_J,$$
then the proof of Theorem~\ref{asymptotics:global} and thus of Theorem~\ref{theo:stab_terrace} is already complete.

However it remains to deal with the case when~$\Phi^c$ contains two or more traveling waves and the remainder of this section is devoted to this situation. We point out that hereafter we keep the notation $\Phi^c$ to highlight the fact that all waves share the same speed~$c$.

As in the previous section, we define
$$\Omega_{c} (u) := \{ u_\infty \, | \ \ \exists t_n \to +\infty, \ u(t+ t_n, x + c t_n ) \to u_\infty (t,x) \mbox{ loc. unif. as } n \to +\infty \},$$
the set of limits of the solution in the moving frame with speed $c$. By Proposition~\ref{prop:trap2} this set is not empty, and by~\eqref{inequality1} and the fact that $\Phi \equiv \Phi^c$, any $u_\infty \in \Omega_c (u)$ must satisfy
\begin{equation}\label{severalfrontslimit}
\Phi^c (t,x; \overrightarrow{\xi}_{sub}) \leq u_\infty (t,x) \leq \Phi^c (t,x ;\overrightarrow{\xi}_{sup}),
\end{equation}
for all $(t,x) \in \mathbb{R}^2$. We will prove that $\Omega_c (u)$ actually is a singleton consisting of a single shift of the terrace solution.

Though it proceeds similarly as in the case of a single traveling wave, our argument here will rely on some induction to deal with the whole terrace solution. Hence we define another partial terrace solution, constituting of the lower $\ell$ waves, as 
\begin{equation}\label{partialterrace}
	\Phi_{\ell}(t,x; \overrightarrow{\xi}_\ell \ ):= \sum_{j=1}^{\ell} (\phi_j (x - \xi_j - c t) - p_{j-1}),
\end{equation}
with the convention that $\overrightarrow{\xi}_\ell = (\xi_1, \cdots, \xi_\ell)$ and $\Phi_0 \equiv 0$. Note that $$\Phi^c (t,x;\overrightarrow{\xi} \ ) = \Phi_J(t,x;\overrightarrow{\xi} \ ).$$
For later use, we also define a partial terrace solution constituting of the upper $J-\ell$ waves, which also rewrites as
$$\Phi^c  (t,x; \overrightarrow{\xi}) - \Phi_\ell (t,x; \overrightarrow{\xi}_l) = \sum_{j=\ell +1}^{J} ( \phi_j (x-\xi_j -ct) - p_{j-1}).$$
Notice that this partial terrace does not depend on the whole vector $ \overrightarrow{\xi}$ but only on the $J-\ell$ last components $(\xi_{\ell +1}, \cdots , \xi_{J})$.

Our next step is to adapt Lemma~\ref{lem:stab_new1} to this context.
\begin{lem}\label{lem:stab_new2}
	Assume that there exist $\overrightarrow{X} = (X_1, \cdots , X_J) \in \R^J$ and $u_\infty \in \Omega_c (u)$ such that
	$$ u_\infty ( 0,x) \geq \Phi^{c} (0,x;\overrightarrow{X}),$$
	for all $x \in \R$, where $\overrightarrow{X}$ is such that $\Phi^c (0,x;\overrightarrow{X})$ is a terrace solution.
	
	Then any $\tilde{u}_\infty \in \Omega_c (u)$ also satisfies
	$$\tilde{u}_\infty (t,x) \geq \Phi^{c} (t,x;\overrightarrow{X}),$$
	for all $(t,x) \in \R^2$.
\end{lem}
\begin{proof}
We will show that for any $\varepsilon >0$, there exists $T>0$ such that
$$u (T, x + c T ) \geq \Phi^c (T, x + cT ; \overrightarrow{X} - \varepsilon)$$
for all $x \in \mathbb{R}$, where for simplicity we denoted $\overrightarrow{X} - \varepsilon = (X_1 - \varepsilon , \cdots, X_J - \varepsilon)$. The wanted conclusion then immediately follows by the weak comparison principle, and taking the large time and $\varepsilon \to 0$ limits. Recall also that $\Phi^c  \equiv \Phi_J $, where $\Phi_l$ ($l=0,\cdots,J$) is defined in \eqref{partialterrace}. Thus we may proceed by induction and show that, for any integer $\ell \in \{ 0 , \cdots, J\}$ and $\varepsilon >0$, there exists $T_\ell >0$ such that
\begin{equation}\label{lem_newclaim2_goodone}
u (T_\ell ,  x + c T_\ell) \geq \Phi_\ell ( T_\ell ,x + cT_\ell; \overrightarrow{X}_\ell - \varepsilon)
\end{equation}
for all $x \in \mathbb{R}$.

When $\ell = 0$, then $\Phi_0 \equiv 0$ and \eqref{lem_newclaim2_goodone} becomes trivial.
Now, fix $\varepsilon >0$ and let the claim~\eqref{lem_newclaim2_goodone} hold for some $\ell \in \left\{ 0,\cdots, J-1 \right\}$. Then, we will show that this claim still holds for $\ell+1$. First, up to some translation, the support of each $\phi_j '$ is an open interval $(0,\eta_j)$. Thus, for any $t \geq 0$ and $x \geq X_{\ell+1} + \eta_{\ell+1} - \varepsilon $, we have
$$\phi_{\ell +1} (x - X_{\ell+1} + \varepsilon ) = p_{\ell+1},$$
and, from the definition of $\Phi_{\ell}$,
$$\Phi_{\ell + 1} (t, x + ct ; \overrightarrow{X}_{\ell+1} - \varepsilon)  =	\sum_{j=1}^{\ell} (\phi_j (x - X_j  + \varepsilon ) - p_{j-1}) = \Phi_\ell (t, x+ct; \overrightarrow{X}_\ell - \varepsilon) .$$
Now, by our induction hypothesis and the weak comparison principle, we get that
$$u (t ,  x + c t) \geq \Phi_{\ell} ( t ,x + c t; \overrightarrow{X}_\ell - \varepsilon)$$
for any $t \geq T_\ell$ and $x \in \mathbb{R}$, hence
\begin{equation}\label{interval1}
u (t ,  x + c t) \geq \Phi_{\ell+1} ( t ,x + c t; \overrightarrow{X}_{\ell+1} - \varepsilon)
\end{equation}
for any $t \geq T_\ell$ and $x \geq X_{\ell+1} + \eta_{\ell+1} - \varepsilon$.

It remains to deal with the left half-line $(-\infty,X_{\ell+1} + \eta_{\ell+1}-\varepsilon)$, which we divide into two sub-intervals:
$$I_1 = [X_{\ell+1}-\ep / 2, X_{\ell+1} +\eta_{\ell+1} - \ep],$$
$$I_2 = (-\infty,X_{\ell+1}-\ep / 2).$$
Note that for the terrace solution to be continuous, we must have $X_j \geq X_{j +1 } + \eta_{j +1}$ for any integer~$j$. Then
\begin{equation}\label{phicphiell}
  \Phi^c (t,x+ct;\overrightarrow{X} - \varepsilon) =  \phi_{\ell+1}(x-X_{\ell+1}+\varepsilon), \qquad
  \Phi^c (t,x+ct;\overrightarrow{X}) \geq  \phi_{\ell+1}(x-X_{\ell+1}),
\end{equation}
for any $t \geq 0$ and $x \in I_1$. Moreover, for $x \in I_1$, by the strict monotonicity of $\phi_{\ell +1}$ in $(0, \eta_{\ell +1})$, we have that
\begin{equation}\label{phicphiell2}
\phi_{\ell +1} (x - X_{l+1} + \varepsilon) < \phi_{\ell +1} (x - X_{l+1}). 
\end{equation}
Recalling our assumption that $u_\infty (0,x) \geq \Phi^c (0,x;\overrightarrow{X})$, which also implies by the comparison principle that $u_\infty (t,x+ct) \geq \Phi^c (t, x+ct ; \overrightarrow{X})$ for $t \geq 0$, we infer from \eqref{phicphiell}-\eqref{phicphiell2} that
$$u_\infty (t, x+ct) > \Phi^c (t,x+ ct; \overrightarrow{X}-\varepsilon),$$
for any $t \geq 0$ and $x \in I_1$. From the definition of $\Omega_c (u)$ and Proposition~\ref{prop:trap2}, there exists a time sequence $t_n$ such that $$u(t+t_n,x+ct+ct_n) \to u_\infty(t,x+ct),$$
as $n\to +\infty$, where the convergence is understood in the locally uniform sense. Thus, there exists a time $t_n > T_\ell$ such that
	\begin{eqnarray}
		u(t+t_n,x+ct+ct_n) &\ge& \Phi^c(t+t_n,x+c(t+t_n); \overrightarrow{X}  - \varepsilon) \nonumber  \\ &\ge& \Phi_{\ell+1} (t+t_n,x+c(t+t_n);\overrightarrow{X}_{\ell+1}- \varepsilon), \label{interval2}
	\end{eqnarray}
in the compact domain $(t,x) \in [0,1] \times I_1$.

Next, we deal with $x \in I_2$. We claim that there exist $n$ arbitrarily large and $\tau \in [0,1]$ such that 
\begin{equation}\label{interval3}
x_{\ell+1}^u(\tau+t_n) \ge c\tau + ct_n + X_{\ell+1} - {\ep \over 2},
\end{equation}
where $x_{\ell+1}^u$ was defined in Proposition~\ref{prop:trap3}. If \eqref{interval3} holds true, then by spatial monotonicity of~$u$ we get that $$u  (\tau+ t_n  , x +c \tau + c t_n) \geq p_{\ell+1} = \Phi_{\ell+1} (t+t_n,x+c(t+t_n);\overrightarrow{X}_{\ell +1} - \varepsilon),$$
for $x \in I_2$. Together with~\eqref{interval1}-\eqref{interval2}, this implies that 
$$u (T_{\ell +1}, c + c T_{\ell +1}) \geq \Phi_{l+1} (T_{\ell+1} , x + c T_{\ell+1} ; \overrightarrow{X}_{\ell +1} - \varepsilon)$$
with $T_{\ell +1} = \tau + t_n$, i.e.~\eqref{lem_newclaim2_goodone} holds true at rank $\ell +1$ and by induction we reach the wanted conclusion.

Thus it only remains to check \eqref{interval3}. Since this proceeds exactly as in the proof of Lemma~\ref{lem:stab_new1}, we will omit some details. We proceed by contradiction and assume that
	\begin{equation} \label{contra}
		x_{\ell+1}^u (t + t_n) <  c t + c t_n - X_{\ell+1} - \frac{\varepsilon}{2},
	\end{equation}
	for any large~$n$ and $t \in [0,1]$. Then the convergence of $u (t+t_n , x + ct +ct_n)$ to $u_\infty \geq \Phi^c (\cdot, \cdot; \overrightarrow{X})$ together with~\eqref{contra} imply that 
$$u(t+t_n, x + ct+ ct_n)\in \left( p_{\ell+1} - \frac{1}{n} , p_{\ell+1} \right)$$ for all $t \in [0,1]$ and $x \in I := [-X_{\ell +1} - \varepsilon/2, -X_{\ell +1}]$. Integrating \eqref{eq:rd} on the same subdomain and passing to the limit as $n \to +\infty$,, we get that
	$$\liminf_{n \to \infty}  \int_I u(t_n+1,x+c(t_n+1)) dx \ge {\varepsilon \over 2} (f^* + p_{\ell+1}  > {\varepsilon \over 2} p_{\ell+1},$$ where $0 < f^* < \inf_{(1-1/n,1)}f$. Thus, $u(t_n+1,x+c(t_n+1))$ exceeds $p_{\ell +1}$ for some $x \in I$, which contradicts~\eqref{contra}. Thus~\eqref{interval3} holds true and this completes the proof of Lemma~\ref{lem:stab_new2}.	
\end{proof}
Now recall~\eqref{severalfrontslimit}, i.e. that all $u_\infty \in \Omega_c (u)$ are framed between the same two shifts of the terrace solution $\Phi^c$. Thus we can introduce the critical shifts $\overrightarrow{X}_- = (X_{1}^- , \cdots , X_{J}^-)$ and $\overrightarrow{X}_+ = (X_{1}^+ , \cdots , X_{J}^+)$, where each component of $\overrightarrow{X}_-$ can be defined recursively as 
\begin{eqnarray}
	\nonumber && X_{1}^- := \inf \{ X \, | \ \  \forall u_\infty \in \Omega_{c} (u),  \  \forall (t,x) \in \R^2, \ \Phi_{1} (t,x;  -X) \leq u_\infty (t,x) \} , \\
	\nonumber && X_{2}^- := \inf \{ X \, | \ \  \forall u_\infty \in \Omega_{c} (u),  \  \forall (t,x) \in \R^2, \ \Phi_{2} (t,x; ( -X_{1}^-,-X)) \leq u_\infty (t,x) \} , \\
	\nonumber && \cdots \\
	\nonumber && X_{J}^- :=  \inf \{ X \, | \ \  \forall u_\infty \in \Omega_{c} (u),  \  \forall (t,x) \in \R^2, \ \Phi_{J} (t,x; ( -X_{1}^-,\cdots,-X)) \leq u_\infty (t,x) \} ,
\end{eqnarray}
and each component of $\overrightarrow{X}_+$ as
\begin{eqnarray}
	\nonumber && X_{J}^+ :=  \inf \{ X \, | \ \  \forall u_\infty \in \Omega_{c} (u),  \  \forall (t,x) \in \R^2, \ (\Phi^c - \Phi_{J-1}) (t,x; X)  \geq u_\infty (t,x) \} \\
	\nonumber && X_{J-1}^+ :=  \inf \{ X \, | \ \  \forall u_\infty \in \Omega_{c} (u),  \  \forall (t,x) \in \R^2, \ ( \Phi^c - \Phi_{J-2}) (t,x; (X, X_J^+)   \geq u_\infty (t,x) \} \\
	\nonumber && \cdots \\
\nonumber && X_{1}^+ := \inf \{ X \, | \ \  \forall u_\infty \in \Omega_{c} (u),  \  \forall (t,x) \in \R^2, \ \Phi^c (t,x;  (X,X_2^+, \cdots, X_J^+ )) \geq u_\infty (t,x) \} .
\end{eqnarray}
Now we will show that both terrace functions $\Phi^c (t,x; -\overrightarrow{X}_-)$ and $\Phi^c (t,x;\overrightarrow{X}_+)$ belong to $\Omega_c (u)$, and also that they are terrace solutions (recall Definition~\ref{terracesolution}). It will then follow by Lemma~\ref{lem:stab_new2} that both terraces actually coincide, and in particular $\Omega_c (u)$ reduces to a singleton. Below we only deal with the lower terrace function $\Phi^c (t,x;-\overrightarrow{X}_-)$, since the argument for the upper terrace function $\Phi^{c}(t,x;\overrightarrow{X}_+)$ is symmetrical.

\begin{lem} \label{lowerfront}
	There exists a function $u_\infty \in \Omega_c(u)$ such that 
\begin{equation}\label{onemoretime}
u_\infty(t,x) = \phi_{1}(x-ct+X_{1}^-)
\end{equation}
for all $t \in \mathbb{R}$ and $x \in (ct-X_{1}^-,\infty)$. 

In particular we must have that
\begin{equation}\label{indeedterrace?}
-X_{2}^- - \eta_2 \leq - X_1^-,
\end{equation}
so that 
$$\Phi_2 (t,x; (-X_1^-, -X_2^-))$$
is a terrace solution of \eqref{eq:rd}.
\end{lem}
\begin{rem}
This proof shares some similarities with that of Proposition~\ref{prop:stab_12}, but there are some differences to prepare the ground for our induction. Hence we repeat it for clarity.
\end{rem}
\begin{proof}
	First recall that $spt (\phi_{1} ' ) = (0,\eta_{1})$. In particular $ \phi_{1} (x) = p_{1}$ for all $x \leq 0$ and $\phi (x)= 0$ for all $x \geq \eta_{1}$. Let $\widehat{u}_\infty$ be any function in $\Omega_c (u)$. By the definition of $X_1^-$ and by continuity, we have that
$$\widehat{u}_\infty (t,x) \geq \phi_1 (x -ct + X_1^-),$$
for all $(t,x) \in \mathbb{R}^2$.

Now, we claim that for any $\varepsilon >0$, we have
	\begin{equation}\label{eq:lem_limit3}
		\inf_{(t,x) \in \R^2} \widehat{u}_\infty (t,x) - \phi_{1} (x - c t + X_{1}^- - \varepsilon ) < 0.
	\end{equation}
Indeed, if \eqref{eq:lem_limit3} does not hold, then
$$\widehat{u}_\infty (t,x) \geq \phi_1 (x-ct + X_1^- - \varepsilon)$$
for all $(t,x) \in \mathbb{R}^2$. Together with \eqref{severalfrontslimit}, we get that
$$\widehat{u}_\infty (t,x) \geq \Phi^c (t,x; \overrightarrow{X} ),$$
where $\overrightarrow{X} = (-X_1^- + \varepsilon, \xi_{sub,2}, \cdots, \xi_{sub,k} ) $. It also follows from~\eqref{severalfrontslimit} and the definition of $X_1^-$ that 
$$-X_1^- + \varepsilon  \geq - X_1^- \geq \xi_{sub,1}.$$
In particular, since $\Phi^c (\cdot, \cdot ; \overrightarrow{\xi}_{sub})$ is a terrace solution, then so is $\Phi^c (\cdot, \cdot ; \overrightarrow{X})$. More precisely, this inequality ensures that this new shift preserves the needed regularity for a terrace function to solve~\eqref{eq:rd}. Therefore we can apply Lemma~\ref{lem:stab_new2}, and we get that any other $\tilde{u}_\infty \in \Omega_c (u)$ also satisfies that
$$\tilde{u}_\infty (t,x) \geq \Phi^c (t,x; \overrightarrow{X} ) \geq \Phi_1 (t,x; -X_1^- + \varepsilon),$$
for all $(t,x) \in \mathbb{R}^2$.
This contradicts the definition of $X_1^-$, hence \eqref{eq:lem_limit3} holds true for any $\widehat{u}_\infty \in \Omega_c (u)$.

Next we further claim that
	\begin{equation}\label{eq:lem_limit4}
		\inf_{t \in \R} \inf_{x - ct + X_{1}^- - \varepsilon \in (0,\eta_1)}  \widehat{u}_\infty (t,x) - \phi_{1} (x-ct + X_{1}^- - \varepsilon ) < 0.
	\end{equation}
	Otherwise, we would have that $\widehat{u}_\infty (t,ct - X_{1}^- + \varepsilon) \ge p_{1}$ for all $t \in \R$, hence $\widehat{u}_\infty (t,x) \ge p_{1} = \phi_{1} (x-ct + X_{1}^- - \varepsilon)$ for all $t \in \R$ and $x \leq ct - X_{1}^- + \varepsilon$. Moreover, $\widehat{u}_\infty (t,x) \geq 0 =  \phi_{1} (x -ct + X_{1}^- - \varepsilon)$ for all $t \in \R$ and $x \geq ct - X_{1}^- + \varepsilon+\eta_1$. In other words, \eqref{eq:lem_limit4} must hold not to contradict \eqref{eq:lem_limit3}.
	
	Now, by \eqref{eq:lem_limit4} and for any $n \in \mathbb{N}^*$, there exists some $(s_n, y_n) \in \mathbb{R}^2$ such that
\begin{equation}\label{moreandmore}
\phi_1 (y_n  - c s_n + X_1^-) \leq u_{\infty} (s_n,y_n) < \phi_{1} \left(y_n - c s_n + X_{1}^- - \frac{1}{n} \right) , 
\end{equation}
	and
	$$ c s_n - X_{1}^- + \frac{1}{n} \leq y_n \leq \eta_1 +c s_n - X_{1}^- + \frac{1}{n}.$$
	In particular,
	$$ \sup_{n \in \mathbb{N}}  | y_n - c s_n | < + \infty ,$$
and up to extraction of another subsequence we may assume that
$$y_\infty := \lim_{n \to +\infty} y_n - c s_n \in [-X_1^- , \eta_1 - X_1^-].$$
Furthermore, since the function $\widehat{u}_\infty$ belongs to $\Omega_{c} (u)$, one can also find a sequence $t_n$ such that $t_n + s_n \to +\infty$ and
$$u (\cdot +t_n + s_n, \cdot + ct_n + cs_n ) - \widehat{u}_\infty (\cdot + s_n, \cdot  + c s_n ) \to 0,$$
locally uniformly as $n \to +\infty$. By Proposition~\ref{prop:trap2}, it follows that, possibly up to extraction of another subsequence, both functions
$$u (\cdot + t_n + s_n, \cdot + ct_n+ c s_n ), \quad \widehat{u}_\infty (\cdot + s_n, \cdot + c s_n ) $$
converge locally uniformly as $n \to +\infty$ to some (possibly distinct)
$$u_\infty \in \Omega_c (u).$$
Hence, by the definition of $X_1^-$ we have
$$u_\infty (t,x) \geq \phi_1 (x - ct + X_1^-)$$
for all $(t,x) \in \mathbb{R}^2$, and also
$$u_\infty (0, y_\infty) = \phi_1 (y_\infty +X_1^-)$$
by passing to the limit in \eqref{moreandmore}.

If $y_\infty \in ( - X_{1}^-,  \eta_1 -  X_{1}^- ]$, then we can apply Proposition~\ref{prop:strong} to infer that 
$$u_\infty (t,x) = \phi_1 (x -ct + X_1^-),$$
for all $t \in \mathbb{R}$ and $x \in [ct -X_1^-, ct  - X_1^- + \eta_1] $. Since $u_\infty$ is nonnegative and nonincreasing in space, we find that the same equality also holds for $x \geq  ct - X_1^- +\eta$, and we have reached the wanted conclusion.

It only remains to consider the case when $y_\infty = -X_{1}^-$, which means that $u_\infty (0,  - X_{1}^-) = \phi_1 ( 0) = p_1$, and also $u_\infty (t,x) > \phi_1 (x-ct + X_1^-)$ for all $x \in (ct - X_1^-, ct - X_1^- + \eta]$. As in the proof of Proposition~\ref{prop:stab_12} we apply the Hopf lemma and find that
$$\partial_x u_\infty (0,-X_1^-) >0,$$
which contradicts the fact that any function in $\Omega_c (u)$ must be nonincreasing in space. Hence this case does not happen and the proof of \eqref{onemoretime} is complete.

The inequality~\eqref{indeedterrace?} easily follows from~\eqref{onemoretime}. Indeed, we have that $u_\infty (0,-X_1^-) = \phi_1 (0) = p_1$. By the definition of $X_2^-$ and a continuity argument, we also have that
$$\Phi_2 (t,x; (-X_1^-, -X_2^-)) \leq u_\infty (t,x).$$
At the point $(0,-X_1^-)$, this equality rewrites as
$$\phi_1 (0 ) + \phi_2 ( X_2^- - X_1^-  ) - p_1 \leq p_1 .$$
From $\phi_1 (0) = p_1$, we get that
$$\phi_2 (X_2^- - X_1^-) \leq p_1.$$ 
Since the traveling wave $\phi_2$ connects $p_2$ and $p_1$, this is only possible if 
$$\phi_2 (X_2^- - X_1^- ) = p_1 \ \mbox{ and } \ -X_2^- - \eta_2 \leq - X_1^-,$$
where $(0,\eta_2)$ denotes the support of $\phi_2 '$. Finally, this implies that $\Phi_2 (\cdot,\cdot; (-X_1^- , - X_2^-)$ and its spatial derivative are continuous, and therefore it solves~\eqref{eq:rd}; see also Remark~\ref{rem:solution}. We have reached the wanted conclusion that it is a terrace solution and the lemma is proved.
\end{proof}

\begin{prop}\label{prop:last}
	The terrace solution $\Phi^c(t,x;-\overrightarrow{X}_-)$ is contained in $\Omega_c(u)$.
\end{prop}
\begin{proof}
The argument proceeds by induction. For the sake of brevity we only sketch how to reiterate the proof of Lemma~\ref{lowerfront}. That is we will show that there exists a function $u_\infty \in \Omega_c (u)$ such that
\begin{equation}\label{reallyfinal}
u_\infty (t,x) = \Phi_2 (t,x ; (-X_1^-, - X_2^-)),
\end{equation}
for all $t \in \mathbb{R}$ and $x \in (ct - X_2^-, \infty).$

First we pick $\widehat{u}_\infty \in \Omega_c (u)$ as in Lemma~\ref{lowerfront}, i.e. such that $\widehat{u}_\infty (t,x) = \phi_1 (x -ct + X_1^-)$ for all $t \in \mathbb{R}$ and $x > ct - X_1^-$. Now we claim that, for any $\varepsilon >0$,
\begin{equation}\label{eq:almostfinal}
\inf_{(t,x) \in \mathbb{R}^2} \widehat{u}_\infty (t,x) - \Phi_2 (t,x; (-X_1^- , -X_2^- + \varepsilon))< 0.
\end{equation}
Indeed, assume otherwise that $\widehat{u}_\infty \geq \Phi_2 (\cdot, \cdot ; (-X_1^- , - X_2^- + \varepsilon)$ for some $\varepsilon >0$. Then, due to Lemma~\ref{lowerfront}, we have that
$$p_1 = \phi_1 (0)=\widehat{u}_\infty (0, -X_1^-) \geq \Phi_2 (0,0 ; (-X_1^- , -X_2^- + \varepsilon)) =  \phi_2 (X_2^- - X_1^- - \varepsilon) . $$
Since $\phi_2 \geq p_1$, we get that
$$\phi_2 (X_2^- - X_1^- - \varepsilon) = p_1.$$ 
This in turn implies that $\Phi_2 (\cdot, \cdot ; (-X_1^- , -X_2^- + \varepsilon))$ is a terrace solution of \eqref{eq:rd}. Then, as in the proof of Lemma~\ref{lowerfront}, the fact that $\widehat{u}_\infty \geq \Phi_2 (\cdot, \cdot; (-X_1^-, -X_2^- + \varepsilon))$ and an application of Lemma~\ref{lem:stab_new2} lead to a contradiction with the definition of $-X_2^-$. Thus the claim~\eqref{eq:almostfinal} holds true.

Next, due to Lemma~\ref{lowerfront}, we already know that $\widehat{u}_\infty$ coincides with $\Phi_1 (\cdot, \cdot ; -X_1^-)$ on a half space, thus we infer from \eqref{eq:almostfinal} that
$$\inf_{t \in \mathbb{R}}  \inf_{x \leq ct - X_1^-} \widehat{u}_\infty (t,x) - \Phi_2 (t,x; (-X_1^-, -X_2^- + \varepsilon)) < 0.$$
Recall that $(0,\eta_2)$ denotes the support of $\phi_2'$. Then we proceed exactly as before and find some sequences $s_n$ and $y_n$ such that $y_n - c s_n \to y_\infty \in [-X_2^-, \eta_2 - X_2^-]$, as well as some $u_\infty \in \Omega_c (u)$ such that
$$\widehat{u}_\infty (\cdot + s_n, \cdot + c s_n) \to u_\infty,$$
where the convergence is understood in the locally uniform sense as $n \to +\infty$, and lastly
\begin{equation}\label{maybefinal}
u_\infty (0, y_\infty) = \phi_2 ( y_\infty + X_2^-).
\end{equation}
As in the proof of Lemma~\ref{lowerfront}, it follows from Proposition~\ref{prop:strong} and a Hopf lemma that
$$u_\infty (t,x) = \phi_2 (x - c t + X_2^-),$$
for all $t \in \mathbb{R}$ and $x \in [ct - X_2^- , ct - X_2^- + \eta_2]$. On the other hand, we get from~\eqref{maybefinal} and the fact that $\widehat{u}_\infty$ satisfies~\eqref{onemoretime} that
$$u_\infty (t,x) = \phi_1 (x - ct + X_1^-),$$
for all $t \in \mathbb{R}$ and $x > ct - X_1^-$. Putting together these two facts, and recalling also~\eqref{indeedterrace?} and that any function in $\Omega_c (u)$ is nondecreasing in space, we conclude that 
$$u_\infty (t,x) = \Phi_2 (t,x ; (-X_1^- , - X_2^-))$$
for all $t \in \mathbb{R}$ and $x \geq ct - X_2^- $. In other words, we have found $u_\infty \in \Omega_c (u)$ such that~\eqref{reallyfinal} holds true. Reiterating this argument, one may end the proof of Proposition~\ref{prop:last}.
\end{proof}
As mentioned before, a symmetrical argument shows that that the terrace solution $\Phi^c (t,x; \overrightarrow{X}_+)$ also belongs to $\Omega_c (u)$. Finally, by Lemma~\ref{lem:stab_new2}, we conclude that $-\overrightarrow{X}_- = \overrightarrow{X}_+$, so that actually $\Omega_c (u)$ reduces to a single terrace solution. Theorems~\ref{asymptotics:global} and~\ref{theo:stab_terrace} are now proved.

\section*{}

\noindent{\bf Acknowledgements}

\noindent{This work was carried out in the framework of the CNRS International Research Network ``ReaDiNet''. The two authors were also supported by the joint PHC Star project MAP, funded by the French Ministry for Europe and Foreign Affairs and the National Research Fundation of Korea. The first author also acknowledges support from ANR via the project Indyana under grant agreement ANR- 21- CE40-0008.}

 \end{document}